\documentclass[12pt,english,a4paper]{smfart}

\usepackage[utf8]{inputenc}
\usepackage[T1]{fontenc}
\usepackage{lmodern}
\usepackage{fixcmex}
\usepackage{smfthm}
\usepackage[headings]{fullpage}
\usepackage{amssymb}
\usepackage[all]{xy}

\let\cal\mathcal

\let\hat\widehat
\let\tilde\widetilde
\let\phi\varphi

\let\epsilon\varepsilon

\newcommand\Q{{\bf Q}} 
\newcommand\Z{{\bf Z}}
\newcommand\C{{\bf C}}
\newcommand\N{{\bf N}}
\newcommand\R{{\bf R}}
\newcommand\A{{\bf A}}
\newcommand\E{{\bf E}}
\newcommand\B{{\bf B}}
\newcommand\G{{\mathcal G}}
\newcommand\F{{\bf F}}

\newcommand\D{{\bf D}}

\renewcommand\c{{\bf c}}

\newcommand\id{{\mathrm{id}}}

\newcommand\pa{{\mathrm{pa}}}

\newcommand\an{{\mathrm{an}}}

\newcommand\ra{{\rightarrow}}

\renewcommand{\O}{{{\cal O}}}

\newcommand\Zp{{\Z_p}}
\newcommand\Qp{{\Q_p}}
\newcommand\Cp{{\C_p}}

\newcommand\Qpbar{{\overline{\Q}_p}}

\newcommand\At{{\tilde{\bf{A}}}}
\newcommand\Atplus{{\tilde{\bf{A}}^+}}
\newcommand\Bt{{\tilde{\bf{B}}}}
\newcommand\Btplus{{\tilde{\bf{B}}^+}}
\newcommand\Et{{\tilde{\bf{E}}}}
\newcommand\Etplus{{\tilde{\bf{E}}^+}}

\newcommand\Gal{{\mathrm{Gal}}}

\newcommand\LT{{\mathrm{LT}}}
\newcommand\la{{\mathrm{la}}}
\newcommand\cbf{{\bf{c}}}
\newcommand\kbf{{\bf{k}}}

\newcommand\cycl{{\mathrm{cycl}}}

\newcommand\rig{{\mathrm{rig}}}

\newcommand\valr{{\mathrm{val}_R}}
\newcommand\valM{{\mathrm{val}_M}}
\newcommand\ac{{\mathrm{ac}}}

\numberwithin{equation}{section}

\author{Léo Poyeton}
\address{Institut de Mathématiques de Bordeaux}
\email{leo.poyeton@math.u-bordeaux.fr}
\urladdr{https://www.math.u-bordeaux.fr/~lpoyeton/}

\date{\today}

\title{Locally analytic vectors and $\Zp$-extensions}

\begin{document}

\begin{abstract}
Let $K$ be a finite extension of $\Qp$ and let $\G_K = \Gal(\Qpbar/K)$. Lately, interest has risen around a generalization of the theory of $(\phi,\Gamma)$-modules, replacing the cyclotomic extension with an arbitrary infinitely ramified $p$-adic Lie extension. Computations from Berger suggest that locally analytic vectors should provide such a generalization for any arbitrary infinitely ramified $p$-adic Lie extension, and this has been conjectured by Kedlaya.

In this paper, we focus on the case of $\Zp$-extensions, using recent work of Berger-Rozensztajn and Porat on an integral version of locally analytic vectors and explain what the locally analytic vectors in the higher rings of periods $\At^{\dagger}$ look like in this setting. We show that the existence of nontrivial locally analytic vectors in $\At^\dagger$, a necessary condition for Kedlaya's conjecture to hold, is equivalent to the existence of an overconvergent lift of the field of norms attached to the $\Zp$-extension. 

In the anticyclotomic setting, assuming that such an overconvergent lift exists, we are able to construct elements in the corresponding Robba ring which should not exist according to a conjecture of Berger. We then prove that in this specific setting, a particular case of Berger's conjecture holds, discarding the existence of such elements. In particular, this disproves Kedlaya's conjecture and shows that there is no overconvergent lift of the field of norms in the anticyclotomic setting. 
\end{abstract}



\maketitle

\tableofcontents

\section*{Introduction}
Let $p$ be a prime and let $K$ be a finite extension of $\Qp$. One of the main ideas to study $p$-adic representations and $\Zp$-representations of $\G_K=\Gal(\Qpbar/K)$ is to use an intermediate extension $K \subset K_\infty \subset \Qpbar$ such that $K_\infty/K$ is nice enough but still deeply ramified (in the sense of \cite{coatesgreenberg}), so that $\Qpbar/K_\infty$ is almost étale and ``contains almost all the ramification of the extension $\Qpbar/K$''. If $K_\infty/K$ is an infinitely ramified $p$-adic Lie extension then those assumptions are satisfied. Classically, one lets $K_\infty$ be the cyclotomic extension $K(\mu_{p^\infty})$ of $K$. 

One striking result following this idea has been the construction of cyclotomic $(\phi,\Gamma)$-modules.  Fontaine has constructed in \cite{Fon90} an equivalence of categories $V \mapsto \D(V)$ between the category of all $p$-adic representations of $\G_K$ and the category of étale $(\phi,\Gamma)$-modules. Different theories of cyclotomic $(\phi,\Gamma)$-modules can be defined: one can define them over a $2$-dimensional local ring $\B_K$, over a subring $\B_K^\dagger$ of $\B_K$ consisting of so-called overconvergent elements, or over the Robba ring $\mathcal{R}_K$. In each case, a $(\phi,\Gamma)$-module is a finite free module over the corresponding ring, equipped with semilinear actions of $\phi$ and $\Gamma = \Gal(K_\cycl/K)$ commuting one to another (the ring itself being equipped with such actions). 

Thanks to a theorem of Cherbonnier and Colmez \cite{cherbonnier1998representations} and a theorem of Kedlaya \cite{slopes}, these different theories are equivalent. Moreover, the theories over both $\B_K^\dagger$ and $\B_K$ come with their integral counterparts, so that free $\Zp$-representations of $\G_K$ are equivalent to étale $(\phi,\Gamma)$-modules over some integral subring of either $\B_K^\dagger$ or $\B_K$.

Lately, there has been an increasing interest in generalizing both $(\phi,\Gamma)$-modules theory \cite{Ber13lifting,Car12,KR09} and more generally in understanding how to replace the cyclotomic extension by an arbitrary infinitely ramified $p$-adic Lie extension in $p$-adic Hodge theory \cite{Ber14SenLa, poyeton2022locally}.  

One could try to define $(\phi,\Gamma)$-modules attached to an almost totally ramified $p$-adic Lie extension by copying the constructions in the cyclotomic case. This strategy relies on finding a ``lift of the field of norms'' and happens to work in the Lubin-Tate setting \cite{KR09}. Under some strong assumptions (which are not always met even in the cyclotomic case), namely that the lift is of ``finite height'', Berger showed in \cite{Ber13lifting} that there were some restrictions on the kind of extensions one could consider in this case, proving for example that there is no finite height lift of the field of norms in the anticyclotomic setting. The author proved that, under the same strong assumptions, the only extensions for which one could lift the field of norms were actually only the Lubin-Tate ones \cite{P19} (upto finite extensions). A more natural and less constraining assumption would be to ask for which extensions one could have an overconvergent lift, but in this case almost nothing is known.

An other idea to generalize $(\phi,\Gamma)$-modules theory, and which has been used with success by Berger and Colmez \cite{Ber14SenLa} to generalize Sen theory, has been to use the theory of locally analytic vectors, initially introduced by Schneider and Teitelbaum \cite{schneiderteitelvrai}. Berger and Colmez have shown that Sen theory could be completely generalized to any arbitrary infinitely ramified $p$-adic Lie extension by using locally analytic vectors. Computations from Berger \cite{Ber14MultiLa} showed that locally analytic vectors in the cyclotomic setting recovered the cyclotomic $(\phi,\Gamma)$-modules over the Robba ring, and suggested that the theory of locally analytic vectors should be able to define a theory of $(\phi,\Gamma)$-modules for any arbitrary infinitely ramified $p$-adic Lie extension. In \cite{kedlaya2013conj}, Kedlaya conjectured that indeed, locally analytic vectors should provide a nice $(\phi,\Gamma)$-module theory for any such $p$-adic Lie extension, and that the theory  should even be defined at an integral level.

Up until recently, locally analytic vectors were only defined in a setting in which $p$ is inverted, so it was difficult to use them in an integral setting or in characteristic $p$. One could say that an element $x$ in a free $\Zp$-algebra is locally analytic if it becomes locally analytic after inverting $p$, which is what Kedlaya does in the statement of his conjecture, but this definition is not very practical and does not extend for characteristic $p$ algebras.

Recently, Berger-Rozensztajn \cite{BerRozdecompletion,berger2022super}, Gulotta \cite{gulotta2019equidimensional}, Johansson and Newton \cite{johansson2019extended} and Porat \cite{porat2024locally} have generalized the classical notion of locally analytic vectors (denoted by ``Super-Hölder vectors'' in the papers of Berger and Rozensztajn) to a characteristic $p$ and integral setting, by translating the property of being locally analytic in terms of Mahler expansions. In \cite{porat2024locally}, Porat has proven that these new integral locally analytic vectors can be used to recover cyclotomic $(\phi,\Gamma)$-modules, thus generalizing the computations of Berger \cite{Ber14MultiLa} to an integral setting. This makes it possible to reinterpret Kedlaya's conjecture in terms of those new integral locally analytic vectors. 

In this paper, we focus on the particular case of $\Zp$-extensions, and try to give a description on what the locally analytic vectors in the rings used to define $(\phi,\Gamma)$-modules are. For technical reasons, we have to assume that $p \neq 2$ and we do so from that point on (though we expect all our results to also hold when $p=2$). Let $K_\infty/K$ be a totally ramified $\Zp$-extension, and we look at the structure of the ring $(\At^\dagger)^{\Gal(\Qpbar/K_\infty), \Gal(K_\infty/K)-\la}$, which we write $(\At_K^\dagger)^\la$ for the rest of the introduction. 

Our first result is that only two very different situations may occur:

\begin{theo}
\label{main theo struc}
\begin{enumerate}
\item Either there is no nontrivial locally analytic vectors in $\At_K^{\dagger}$, that is $(\At_K^\dagger)^\la = \O_K$; 
\item or $(\At_K^\dagger)^\la = \phi^{-\infty}(\A_K^\dagger)$, where $\A_K^\dagger$ is a ring of overconvergent functions in one variable.
\end{enumerate}
\end{theo}

In the second case, we also prove that everything behaves as in the cyclotomic setting. In particular, we obtain an overconvergent lift of the field of norms, and we also prove that the existence of such a lift guarantees that we are in the second case of theorem \ref{main theo struc}:

\begin{theo}
\label{theo intro oclift equals nontrivial locan}
If $K_\infty/K$ is a $\Zp$-extension, then there exists an overconvergent lift of the field of norms of $K_\infty/K$ if and only if there exists a nontrivial locally analytic vector in $\At_K^{\dagger}$.
\end{theo}

Of course, if one believes in Kedlaya's conjecture, then the first situation in theorem \ref{main theo struc} above should never arise (otherwise the $(\phi,\Gamma)$-modules theory it defines is too small to encode all the $p$-adic representations of $\G_K$) and we prove that indeed the conjecture implies the existence of nontrivial locally analytic vectors.

When $K= \Q_{p^2}$, the unramified extension of $\Q_p$ of degree $2$, and $K_\infty/K$ is the anticyclotomic extension, there exists an element in $\mathrm{Frac}(\Bt_{\rig,K}^\dagger)$ which is invariant by the Frobenius $\phi_q$, and over which the Galois action is locally analytic. Assuming the existence of a nontrivial lift of the field of norms of $K_\infty/K$, we prove that this implies the existence of a nontrivial element in $(\mathrm{Frac}\cal{R}_K)^{\phi_q=1}$, where $\cal{R}_K$ is the corresponding Robba ring endowed with a Galois action and a Frobenius map. According to a conjecture of Berger in \cite{berger2022substitution}, such an element should not exist. Using the additional tools given by the existence of an overconvergent lift of the field of norms and the properties satisfied by our element in $\mathrm{Frac}(\cal{R}_K)^{\phi_q=1}$, we prove that indeed such an element cannot exist, which proves the following:

\begin{theo}
There is no overconvergent lift of the field of norms in the anticyclotomic setting. In particular, Kedlaya's conjecture does not hold in general.
\end{theo}

In particular, this explains why one would need to use a theory of derived anticyclotomic $(\phi,\Gamma)$-modules in order to have a nice theory: the anticyclotomic $(\phi,\Gamma)$-modules are not in general concentrated in degree $0$. 

\section*{Acknowledgements}
Special thanks to Gal Porat for reading through several versions of this paper. His many comments and remarks were very helpful in improving the quality of the exposition and fixing many typos and problems. I also thank Laurent Berger for several comments on the earlier versions of this paper.

\section*{Structure of the paper}
The first section defines the notations attached to Lubin-Tate extensions that are in use in the rest of the paper. Section 2 recalls the classical theory of locally analytic vectors for $p$-adic Banach spaces and its recent generalization in an integral setting using Mahler expansions. In section 3, we recall the definition of some rings of periods that will be used in the rest of the paper, along with some of their properties, and recall the statement of Kedlaya's conjecture. Section 4 recalls several results regarding holomorphic functions. In section 5, we recall what an overconvergent lift of the field of norms is, and prove that the existence of such a lift implies the existence of nontrivial locally analytic vectors in $\At_K^\dagger$. Section 6 is devoted to the proof of the reverse, namely that the existence of nontrivial locally analytic vectors in $\At_K^\dagger$ implies the existence of such a lift, in the special case where $K_\infty/K$ is a $\Z_p$-extension. In section 7, we construct the Lie logarithm of the Galois action and prove that it has some technical properties related to the kernel of the theta map. In section 8 we prove that the assumption that nontrivial locally analytic vectors exist for the anticyclotomic extension implies the existence of a counterexample of Berger's conjecture, and prove that this counterexample does not actually exist. Finally, we explain in section 9 how this disproves Kedlaya's conjecture and explain that our results imply the existence of nontrivial higher locally analytic vectors in the cases where $(\At_K^{\dagger})^\la = \O_K$. 

\section*{Notations}
For the rest of the paper, we fix a prime $p \neq 2$ and we let $K$ be a finite extension of $\Qp$, with residue field $k_K$ of cardinal $q=p^h$, and ramification index $e$. We let $\O_K$ denote the ring of integers of $K$, and we let $\pi$ be a uniformizer of $\O_K$. 

\section{Lubin-Tate extensions and some specific subextensions}
\label{LubinTate}
Let $\LT$ be a Lubin-Tate formal $\O_K$-module attached to the uniformizer $\pi$ of $\O_K$. For $a \in \O_K$, we let $[a](T)$ denote the power series giving the multiplication by $a$ map on $\LT$. Let $T$ be a local coordinate on $\LT$ such that $[\pi](T) = T^q+\pi T$, except in the particular case where $K=\Qp$ and $\pi=p$, where we choose instead a local coordinate $T$ such that $[p](T) = (1+T)^p-1$. We let $K_n = K(\LT[\pi^n])$ be the extension of $K$ generated by the $\pi^n$-torsion points of $\LT$, and we let $K_\LT = \cup_{n \geq 1}K_n$.  We let $\Gamma_\LT = \Gal(K_\LT/K)$ and $H_\LT = \Gal(\Qpbar/K_\LT)$. By Lubin-Tate theory (see \cite{LT65}), if $g \in \Gamma_\LT$ then there exists a unique $a_g \in \O_K^\times$ such that $g$ acts on the torsion points of $\LT$ through the power series $[a_g](T)$, and the map $\chi_\pi : g \in \Gamma_\LT \mapsto a_g \in \O_K^\times$ is a group isomorphism called the Lubin-Tate character attached to $\pi$.

There is an unramified character $\eta : \G_K \ra \Z_p^\times$ such that $\mathrm{N}_{K/\Qp}(\chi_\pi) = \eta \chi_\cycl$. We let $K_\infty^\eta = \overline{Q}_p^{\ker(\eta \chi_\cycl}$, and we have that $K_\infty^\eta \subset K_\LT$, and $\eta\chi_\cycl$ identifies $\Gal(K_\infty/K)$ with an open subgroup of $\Z_p^\times$. We call extensions of the form $K_\infty^\eta$ twisted cyclotomic extensions.  

For $n \geq 1$, we let $\Gamma_n = \Gal(\LT/K_n)$ so that $\Gamma_n = \{g \in \Gamma_\LT, \chi_\pi(g) \in 1+\pi^n\O_K\}$. We let $u_0 = 0$ and for $n \geq 1$ we let $u_n \in \Qpbar$ be such that $[\pi](u_n)=u_{n-1}$, with $u_1 \neq 0$. We have $K_n = K(u_n)$, and $u_n$ is a uniformizer of $K_n$. We also let $Q_n(T)$ be the minimal polynomial of $u_n$ over $K$, so that $Q_0(T)=T$, $Q_1(T) = [\pi](T)/T$ and $Q_{n+1}(T) = Q_n([\pi](T))$ if $n \geq 1$. 

We let $\log_{\LT} = T + O(\mathrm{deg} \geq 2) \in K[\![T]\!]$ denote the Lubin-Tate logarithm map, which converges on the open unit disk and is such that $\log_{\LT}([a](T))= a \cdot \log_{\LT}(T)$ for $a \in \O_K$. We recall that $\log_{\LT}(T) = T \cdot \prod_{k \geq 1}Q_k(T)/\pi$, and we let $\exp_{\LT}$ denote the inverse of $\log_{\LT}$. 

When $K= \Q_{p^2}$, the unramified extension of $\Qp$ of degree $2$, and $\pi = p$, then $K_\LT$ contains two special and particularly interesting sub-$\Zp$-extensions: the cyclotomic extension $K_{\cycl}=K(\mu_{p^\infty})$ of $K$, which is Galois and abelian over $\Qp$, and the anticyclotomic extension $K_{\mathrm{ac}}$ which is the unique $\Zp$-extension of $K$ which is Galois and pro-dihedral over $\Qp$: the Frobenius $\sigma$ of $\Gal(K/\Qp)$ acts on $\Gal(K_{\mathrm{ac}}/K)$ by inversion. It is linearly disjoint from $K_{\cycl}$ over $K$, and the compositum $K_\cycl \cdot K_\ac$ is equal to $K_\LT$. If we let $\chi_p$ denote the Lubin-Tate character corresponding to $K_\LT$, then $\chi_\cycl = N_{K/\Qp}(\chi_p) = \sigma(\chi_p)\cdot \chi_p$. One defines an anticyclomic character $\chi_\ac : \Gal(K_\ac/K) \ra \O_K^\times$ by $g \mapsto \frac{\chi_p(g)}{\sigma(\chi_p(g))}$ which is an isomorphism on to its image, and the anticyclotomic extension is the subfield of $K_\LT$ fixed by the elements $g \in \Gamma_\LT$ such that $\chi_\ac(g)=1$.

\section{Locally analytic and super-Hölder vectors}
In this section, we recall the classical notion of locally analytic vectors, following \cite{emerton2004locally} and \cite[\S 2]{Ber14MultiLa}, along with the notion of locally analytic vectors for $\Zp$-Tate algebras, following Porat \cite{porat2024locally}. 

Let $G$ be a $p$-adic Lie group, and let $G_0$ be an open subgroup of $G$ which is a uniform pro-$p$-group (see \S 4 of \cite{analyticprop} for the definition of a uniform prop-$p$-group and Interlude A of ibid for the statement). The main interest of such a subgroup $G_0$ is that it provides a nice specific fundamental system of open neighborhoods of $G$, along with coordinates $\c : G_0 \ra \Z_p^d$, where $d$ is the dimension of $G$ as a $p$-adic Lie group. Namely, if we let $G_i = \{g^{p^i}, g \in G_0\}$ then we have the following properties (see \S 4 of \cite{analyticprop} for the proof):
\begin{enumerate}
\item for $i \geq 0$, $G_i$ is an open normal uniform subgroup of $G_0$;
\item $[G_i:G_{i+1}] = p^d$;
\item there is a coordinate $\cbf : G_0 \ra \Z_p^d$ such that for $i \geq 0$, $\cbf(G_i) = (p^i\Zp)^d$;
\item For $g,h \in G_0$, we have $gh^{-1} \in G_i$ if and only if $\cbf(g)-\cbf(h) \in (p^i\Zp)^d$.
\end{enumerate}  

In the rest of this article, if $G$ is a $p$-adic Lie group then we assume that we also have chosen such a subgroup $G_0$, along with coordinates $\c$ and the $(G_i)_{i \geq 0}$ as a fundamental system of open neighborhoods of $G$. 

Let $H$ be an open subgroup of $G$ which is uniform pro-$p$, with coordinate $\cbf : H \to \Z_p^d$. Let $W$ be a $\Qp$-Banach representation of $G$. We say that $w \in W$ is an $H$-analytic vector if there exists a sequence $\left\{w_{\kbf}\right\}_{\kbf \in \N^d}$ such that $w_{\kbf} \rightarrow 0$ in $W$ and such that $g(w) = \sum_{\kbf \in \N^d}\cbf(g)^{\kbf}w_{\kbf}$ for all $g \in H$. We let $W^{H-\an}$ be the space of $H$-analytic vectors. This space injects into $\cal{C}^{\an}(H,W)$, the space of all analytic functions $f : H \to W$.  Note that $\cal{C}^{\an}(H,W)$ is a Banach space equipped with its usual Banach norm, so that we can endow $W^{H-\an}$ with the induced norm, that we will denote by $||\cdot ||_H$. With this definition, we have $||w||_H = \sup_{\kbf \in \N^d}||w_{\kbf}||$ and $(W^{H-\an},||\cdot||_H)$ is a Banach space.

The space $\cal{C}^{\an}(H,W)$ is endowed with an action of $H \times H \times H$, given by
\[
((g_1,g_2,g_3)\cdot f)(g) = g_1 \cdot f(g_2^{-1}gg_3)
\]
and one can recover $W^{H-\an}$ as the closed subspace of $\cal{C}^{\an}(H,W)$ of its $\Delta_{1,2}(H)$-invariants,  where $\Delta_{1,2} : H \to H \times H \times H$ denotes the map $g \mapsto (g,g,1)$ (see \cite[§3.3]{emerton2004locally} for more details).

We say that a vector $w$ of $W$ is locally analytic if there exists an open subgroup $H$ as above such that $w \in W^{H-\an}$. Let $W^{\la}$ be the space of such vectors, so that $W^{\la} = \varinjlim_{H}W^{H-\an}$, where $H$ runs through a sequence of open subgroups of $G$. The space $W^{\la}$ is naturally endowed with the inductive limit topology, so that it is an LB space. 

Let $W$ be a Fréchet space whose topology is defined by a sequence $\left\{p_i\right\}_{i \geq 1}$ of seminorms. Let $W_i$ be the Hausdorff completion of $W$ at $p_i$, so that $W = \varprojlim\limits_{i \geq 1}W_i$. The space $W^{\la}$ can be defined but as stated in \cite{Ber14MultiLa} and as showed in \S 7 of \cite{poyeton2022locally}, this space is too small in general for what we are interested in, and so we make the following definition, following \cite[Def. 2.3]{Ber14MultiLa}:

\begin{defi}
If $W = \varprojlim\limits_{i \geq 1}W_i$ is a Fréchet representation of $G$, then we say that a vector $w \in W$ is pro-analytic if its image $\pi_i(w)$ in $W_i$ is locally analytic for all $i$. We let $W^{\pa}$ denote the set of all pro-analytic vectors of $W$. 
\end{defi}

We extend the definition of $W^{\la}$ and $W^{\pa}$ for LB and LF spaces respectively.

Because the classical definition of locally analytic vectors involves denominators in $p$, it may seem difficult to generalize this notion for $\Zp$-algebras where $p$ is not invertible (and may even be $0$). The main idea to generalize the classical notion of locally analytic vectors to this setting is (as often in $p$-adic analysis) to replace Taylor expansions with Mahler expansions, using binomial coefficients. This is explained and used in \cite{BerRozdecompletion} and \cite{porat2024locally}. Following those two papers, we place ourselves in the following setting: $R$ is a $\Zp$-algebra, which is a Tate ring endowed with a valuation $\valr : R \ra ]-\infty,\infty]$ satisfying the following properties:
\begin{enumerate}
\item $\valr(x) = \infty$ if and only if $x=0$ (meaning that $R$ is separated for the topology induced by $\valr$);
\item $\valr(xy) \geq \valr(x)+\valr(y)$ for all $x,y \in R$;
\item $\valr(x+y) \geq \inf(\valr(x),\valr(y))$ for all $x,y \in R$;
\item $\valr(p) > 0$.  
\end{enumerate}

We extend this definition to $R$-modules. 

In what follows, $G$ is a uniform pro-$p$-group. For an $R$-module $M$, endowed with a compatible valuation $\valM$, we write $\cal{C}^0(G,M)$ for the set of continuous functions from $G$ to $M$.

Following \cite{porat2024locally}, we make the following definition:

\begin{defi}
\label{defi locana via mahler}
\begin{enumerate}
\item Let $\lambda, \mu \in \R$. We let $\cal{C}^{\an-\lambda,\mu}(\Z_p^d,M)$ denote the set of functions $f: \Z_p^d \ra M$ such that $\valM(a_{\underline{n}}(f)) \geq p^\lambda\cdot p^{\lfloor \log_p(|\underline{n}|_\infty) \rfloor}+\mu$ for every $\underline{n}=(n_1,\cdots,n_d)$ in $\Z_p^d$, where $|\underline{n}|_\infty$ denotes the maximum of the $n_i$. Note that it is contained in $\cal{C}^0(G,M)$ (see \S 2 of \cite{porat2024locally}). 
\item We define $\cal{C}^{\an-\lambda,\mu}(G,M)$ by pulling back along $\cbf : G \ra \Z_p^d$ the definition of $\cal{C}^{\an-\lambda,\mu}(\Z_p^d,M)$.
\item We let $\cal{C}^{\an-\lambda}(G,M)$ denote the set of functions $f: G \ra M$ such that there exists $\mu \in \R$ such that $f \in \cal{C}^{\an-\lambda,\mu}(G,M)$.
\item We let $\cal{C}^{\la}(G,M)$ be the colimit of the cofinal system $\{\cal{C}^{\an-\lambda,\mu}(G,M)\}_{\lambda,\mu}$, or equivalently, of the cofinal system $\{\cal{C}^{\an-\lambda}(G,M)\}_{\lambda}$.
\end{enumerate}
\end{defi}

We refer the reader to \S 2 of \cite{porat2024locally} to see different characterization of those sets of functions. 

We now assume that $G$ is a uniform pro-$p$-group, acting on $M$ by isometries. As in the Banach-space setting, the space $\cal{C}^{0}(G,M)$ is endowed with an action of $G \times G \times G$, given by
\[
((g_1,g_2,g_3)\cdot f)(g) = g_1 \cdot f(g_2^{-1}gg_3)
\]
and we define $M^{G,\la}$ (resp. $M^{G,\lambda-\an}$ resp. $M^{G,\lambda-\an,\mu}$) as the subspace of $\cal{C}^{\la}(G,M)$ (resp. $\cal{C}(G,M)^{G,\lambda-\an}$ resp. $\cal{C}(G,M)^{G,\lambda-\an,\mu}$) of its $\Delta_{1,2}(G)$-invariants,  where $\Delta_{1,2} : G \to G \times G \times G$ denotes the map $g \mapsto (g,g,1)$.

We define the locally analytic vectors of $M$ as the elements of 
$$M^{\la} := \varinjlim\limits_{i}M^{G_i-\la}.$$

As explained in Example 2.1.3 of \cite{porat2024locally}, when $R=\Qp$, $M$ is a $\Qp$-Banach space and we recover the classical locally analytic vectors. We can actually give a more precise statement. Let $\mathrm{LA}_h(\Zp,\Qp)$ be the space of functions $f : \Z_p \ra \Qp$ whose restriction to any ball of the form $a+p^h\Zp$ is the restriction of an analytic function $f_{a,h}$. This is a Banach space with the obvious norm. If $W$ is a $\Qp$-Banach space we define $\mathrm{LA}_h(\Zp,W):= W \hat{\otimes}_{\Qp}\mathrm{LA}_h(\Zp,\Qp)$. Theorem 3 of \cite{amice64} and theorem I.4.7 of \cite{colmez2010fonctions} have the following corollary:

\begin{coro}
\label{coro Amice precise}
If $f \in \cal{C}^0(\Zp,\Qp)$, the following are equivalent:
\begin{itemize}
\item $f \in \mathrm{LA}_h(\Zp,\Qp)$;
\item $f \in \cal{C}^{\an-\lambda}(\Zp,\Qp)$ for all $\lambda > - h - \frac{\log(p-1)}{\log(p)}$.  
\end{itemize}
\end{coro}
\begin{proof}
See the proof of \cite[Coro. I.4.8]{colmez2010fonctions}.
\end{proof}

In particular, if $M$ be a $\Qp$-Banach space on which $G$ acts by isometry, then there exists $\lambda \in \R$ such that $x \in M^{G_0-\an,\lambda}$ if and only if there exists $n \geq 0$ such that $x \in M^{G_n-\an}$ (in the sense of the classical definition). 

Finally, one may define higher locally analytic vectors, coming from the derived functor induced by $M \mapsto M^{\la}$. Once again, we follow Porat \cite[\S 2.3]{porat2024locally} by setting
$$R_{\la}^i(M):=\varinjlim\limits_{j}H^i(G_j,\cal{C}^{\la}(G_j,M)),$$
where the cocycles considered are continuous, and we take the inductive topology on $\cal{C}^{\la}(G_j,M)$ induced from that of its submodules $\cal{C}^{\lambda-\an}(G_j,M)$. These groups form what we call the higher locally analytic vectors of $M$, and if 
$$0 \ra M_1 \ra M_2 \ra M_3 \ra 0$$
is an exact sequence (in the appropriate category) then we have a long exact sequence 
$$0 \ra M_1^{\la} \ra M_2^{\la} \ra M_3^{\la} \ra R_{\la}^1(M_1) \ra \cdots$$

\begin{lemm}
\label{lemma lambda an ring}
If $x,y \in R^{G,\lambda-\an}$ then $xy \in R^{G,\lambda-\an}$. 
\end{lemm}
\begin{proof}
See lemma 3.3.1 of \cite{gulotta2019equidimensional}. 
\end{proof}

\begin{lemm}
\label{lemma invertible loc ana}
If $x \in R^\times \cap R^{\la}$ then $x^{-1} \in R^{\la}$.
\end{lemm}
\begin{proof}
See \cite[Lemma 2.15 (ii)]{porat2024locally}.
\end{proof}

\begin{lemm}
\label{lemma Gulotta cangobackto GammaK}
Let $x \in M^{G_i,\la}$. Then $x \in M^{G_0,\la}$.
\end{lemm}
\begin{proof}
This is a consequence of proposition 3.3.5 of \cite{gulotta2019equidimensional}. Gulotta's proposition tells us that, if $S$ is a set of representatives of $\Zp/p\Zp$ and if $\kappa$ is negative enough, then a continuous function $f : \Z_p^k \ra M$ is $\Z_p^k,\kappa$-analytic if and only if for all $s \in S$, $z \mapsto f(pz+s)$ is $\Z_p^k,(\kappa+1)$-analytic. 

Let $i \geq 1$ and let $S = \{s_1,\ldots,s_r\}$ be a set of representatives of $G_i/G_{i-1}$. If $x \in M$ is such that $x \in M^{G_i,\kappa-\an}$, then since $G$ acts by isometry on $M$,  it means that each of the functions $g \in \Gamma_i \mapsto s_j\cdot g(x)$ is $\Gamma_i,\kappa$-analytic. By applying Gulotta's result, this means that if $\kappa$ is negative enough then $x$ is $G_{i-1},(\kappa-1)$-analytic. Applying this process successively shows that for $\kappa$ negative enough, we have that if $x \in M^{G_i,\kappa-\an}$, then we have that $x \in M^{G_0,(\kappa-i)-\an}$, which proves the claim.
\end{proof}

\section{Locally analytic vectors for classical rings of periods}
In this section we quickly recall the definition of some classical rings of periods, and then recall several results regarding the locally analytic vectors attached to $p$-adic Lie extensions (and especially in the cyclotomic and Lubin-Tate cases) in those rings. We also explain how the normalization of the valuation may affect the ``radius of analyticity'' of the elements considered.

\subsection{Some rings of $p$-adic periods} 
In this section, we recall the definition of some rings of $p$-adic periods, defined in \cite{Fon90,fontaine1994corps}, \cite{Ber02} and \cite{Col02}. We also recall the definitions of some rings of periods attached to Lubin-Tate extensions, which can be specialized to recover the rings appearing in the cyclotomic setting.

We let $\Etplus := \varprojlim\limits_{x \mapsto x^q}\O_{\Cp}/\pi$ be the tilt of $\O_{\Cp}$. It is a perfect ring of characteristic $p$ which is equipped with a valuation $v_{\E}$ coming from the one of $\Cp$, and is complete for this valuation. We let $\Et$ denote the fraction field of $\Etplus$. If $F$ is a subfield of $\Cp$, let $\mathfrak{a}_F^c$ be the set of elements $x$ of $F$ such that $v_K(x) \geq c$, and for any $c > 0$ we identify $\Etplus$ with $\varprojlim_{x \mapsto x^q}\mathcal{O}_{\Cp}/\mathfrak{a}_{\Cp}^c$.

If $\{u_n\}_{n \geq 0}$ are as in \S \ref{LubinTate}, then the sequence $\overline{u}:=(\overline{u_0},\overline{u_1},\cdots) \in (\O_{\Cp}/\pi)^\N$ belongs to $\Etplus$, and we have $v_{\E}(\overline{u}) = q/(q-1)e$. 

We let $\Atplus = \O_K \otimes_{\O_{K_0}}W(\Etplus)$, and $\Btplus = \Atplus[1/\pi]$. We also let $\At = \O_K \otimes_{\O_{K_0}}W(\Et)$ and $\Bt = \At[1/\pi]$. We write $[\cdot]$ for the Teichmüller map. We endow these rings with the Frobenius map $\phi_q = \id \otimes \phi^h$. 

By \S 9.2 of \cite{Col02}, there exists $u \in \Atplus$, whose image is $\overline{u}$, and such that $\phi_q(u) = [\pi](u)$ and $g(u) = [\chi_\pi(g)](u)$ if $g \in \Gamma_K$.  If $K= \Qp$ and $\pi=p$, then $u=[\epsilon]-1$, where $\epsilon \in \Etplus$ is a compatible sequence of $q^n$-th roots of $1$. We let $Q_k = Q_k(u) \in \Atplus$. 

Recall that we have a map $\theta : \Atplus \ra \O_{\Cp}$ which is a ring homomorphism, whose kernel is a principal ideal generated by $\phi_q^{-1}(Q_1)$ or by $[\tilde{\pi}]-[\pi]$ (see proposition 8.3 of \cite{Col02}), where $\tilde{\pi} \in \Etplus$ is a compatible sequence of $q^n$-th roots of $\pi$. In particular, $\phi_q^{-1}(Q_1)/([\tilde{\pi}]-\pi)$ is a unit of $\Atplus$ and so are the elements $Q_k/([\tilde{\pi}]^{q^k}-\pi)$ for all $k \geq 1$. 

Every element of $\Btplus[1/[\overline{u}]]$ can be written as $\sum_{k \gg -\infty}\pi^k[x_k]$, where $(x_k)_{k \in \Z}$ is a bounded sequence of $\Et$. For $r > 0$, we define a valuation $V(\cdot,r)$ on $\Btplus[1/[\overline{u}]]$ by the formula 

\[V(x,r) = \inf\limits_{k \in \Z}\left(\frac{k}{e}+\frac{p-1}{pr}v_{\E}(x_k)\right) \textrm{ if } x = \sum_{k \gg -\infty}\pi^k[x_k].\]

If $I$ is a closed subinterval of $[0,+\infty[$, $I \neq [0,0]$, we let $V(x,I) = \inf_{r \in I, r \neq 0}V(x,r)$ (one can take a look at remark 2.1.9 of \cite{GP18} to understand why we avoid defining $V(\cdot,0)$). We define $\Bt^I$ as the completion of $\Btplus[1/[\overline{u}]]$ for $V(\cdot,I)$ if $0 \not \in I$. If $0 \in I$, we let $\Bt^I$ be the completion of $\Btplus$ for $V(\cdot,I)$. We let $\At^I$ be the ring of integers of $\Bt^I$ for $V(\cdot,I)$. By \S 2 of \cite{Ber02}, we have that $\At^{[r,s]}$ is also the $p$-adic completion of $\Atplus[\frac{p}{[\overline{u}]^r},\frac{[\overline{u}]^s}{p}]$. 

For $k \geq 1$, we let $r_k = p^{kh-1}(p-1)$ and $\rho_k = p^{-kh}$. The map $\theta \circ \phi_q^{-k} : \Atplus \ra \O_{\Cp}$ extends by continuity to $\At^I$, provided that $r_k \in I$, in which case we have that $\theta \circ \phi_q^{-k}(\At^I) \subset \O_{\Cp}$. 

For $r > 0$, we define $\Bt^{\dagger,r}$ the subset of overconvergent elements of ``radius'' $r$ of $\Bt$, by 

$$\Bt^{\dagger,r}=\left\{x = \sum_{k \gg -\infty}\pi^k[x_k] \textrm{ such that } \lim\limits_{k \to +\infty}v_{\E}(x_k)+\frac{pr}{(p-1)e}k =+\infty \right\}.$$

Note that $\Bt^{\dagger,r}$ can naturally be identified with a subring of $\Bt^{[r,r]}$ and we endow it with the valuation $V(\cdot,r)$. We let 

$$\At^{\dagger,r}=\left\{x = \sum_{k \geq 0}\pi^k[x_k] \in \At \cap \Bt^{\dagger,r} \textrm{ such that } \forall k \geq 0, v_{\E}(x_k)+\frac{pr}{(p-1)e}k \geq 0\right\}$$ 

and we also endow it with the valuation $V(\cdot,r)$. Note that $\At^{\dagger,r}$ is also the $p$-adic completion of $\Atplus[\frac{p}{[\overline{u}]^r}]$. If $\rho = \frac{r_0}{r}$, we let $\At^{(0,\rho]}:= \At^{\dagger,r}[1/[\overline{u}]]$. We endow $\At^{(0,\rho]}$ with the valuation $v_\rho$ given by the $[\overline{u}]$-adic valuation, so that $\rho v_\rho = V(\cdot,r)$ and $v_\rho = \frac{r}{r_0}V(\cdot,r)$. Note that $\At^{\dagger,r}$ is the ring of integers of $\At^{(0,\rho]}$ for $v_\rho$ and also for $V(\cdot,r)$. Moreover, for any $\rho > 0$, we have $\At^{(0,\rho]}/(\pi) = \Et$. 

We let $\Bt^{\dagger}:= \cup_{r > 0} \Bt^{\dagger,r}$ and $\At^{\dagger} = \cup_{\rho > 0} \At^{(0,\rho]}$. We shall write $\Bt_\rig^{\dagger,r}$ for $\Bt^{[r,+\infty[}$.

If $S$ is any of the rings above, we let $S_{\LT} = S^{H_{\LT}}$. 

For $\rho > 0$, let $\rho' = \rho \cdot e \cdot p/(p-1)\cdot (q-1)/q$. Note that we have $V(u^i,r) = i/r'$. Let $I$ be a subinterval of $[0,+\infty[$ which is either a subinterval of $]1,+\infty[$ or such that $0 \in I$. Let $f(Y) = \sum_{k \in \Z}a_kY^k$ be a power series with $a_k \in K$ and such that $v_p(a_k)+k/\rho' \to +\infty$ when $|k| \to + \infty$ for all $\rho \in I$. The series $f(u)$ converges in $\Bt^I$ and we let $\B_{\LT}^I$ denote the set of all $f(u)$ with $f$ as above. It is a subring of $\Bt_{\LT}^I$ which is stable under the action of $\Gamma_{\LT}$. The Frobenius map gives rise to a map $\phi_q : \B_{\LT}^I \ra \B_{\LT}^{qI}$. 

We let $\B_{\LT}^{\dagger,r}$ denote the set of $f(u) \in \B_{\rig,\LT}^{\dagger,r}$ such that the sequence $\{a_k\}_{k \in \Z}$ is bounded. This is a subring of $\Bt_{\LT}^{\dagger,r}$. We let $t_\pi:=\log_{\LT}(u)$. 
%We also define $\A_{\LT}^{\dagger,r} = \B_{\LT}^{\dagger,r} \cap \At^{\dagger,r}$. 

\begin{lemm}
\label{lemma v in r not just rho}
An element $x = \sum_{k \geq 0}\pi^k[x_k] \in \At^{\dagger,r}$ is a unit of $\At^{\dagger,r}$ if and only if $v_{\E}(x_0) = 0$ and $V(x-[x_0],r) > 0$. Moreover, if $x \in \At^{\dagger}$ is such that $v_{\E}(x_0) \geq 0$ then :
\begin{enumerate}
\item there exists $r > 0$ such that $x \in \At^{\dagger,r}$ ;
\item there exists $s \geq r$ such that $\frac{x}{[x_0]}$ belongs to $\At^{\dagger,s}$ and is a unit of $\At^{\dagger,s}$. 
\end{enumerate}
\end{lemm}
\begin{proof}
The first statement is \cite[Lemm. 5.9]{colmez2008espaces}. 

For item $1$, let us write $x = \sum_{k=0}^\infty p^k[x_k]$. Since $x \in \At^{\dagger}$, there exists $t > 0$ such that $\frac{k}{e}+\frac{p-1}{pt}v_{\E}(x_k)$ goes to $+\infty$ when $k \rightarrow +\infty$, so that the sequence $(\frac{k}{e}+\frac{p-1}{pt}v_{\E}(x_k))$ is bounded below by some constant $C$. If $C \geq 0$ then $x \in \At^{\dagger,r}$ so the first item is satisfied. Otherwise, it is bounded by $-D$ for some $D > 0$. Then if $s \geq t\cdot(eD+1)$, we have $\frac{k}{e}+\frac{p-1}{ps}v_{\E}(x_k) \geq 0$ for $k \geq 1$, and since $v_{\E}(x_0) \geq 0$, this means that $V(x,s) \geq 0$. 

For item $2$, one uses item $1$ to find $s \geq r$ such that $\frac{x}{[x_0]}$ belongs to $\At^{\dagger,s}$, and then up to increasing again $s$ this element is a unit of $\At^{\dagger,s}$ by the first statement of the lemma. 
\end{proof}

\begin{prop}
\label{prop ker theta}
Let $k \geq 0$. Then for any $0 \leq r \leq (p-1)p^{k-1}$, the kernel of $\theta \circ \phi^{-k}: \Bt^{\dagger,r} \rightarrow \Cp$ is a principal ideal of $\Bt^{\dagger,r}$, generated by $[\tilde{\pi}^{p^k}]-\pi$, and the kernel of $\theta \circ \phi^{-k}: \At^{[r,r]} \rightarrow \O_{\Cp}$ is a principal ideal of $\At^{[r,r]}$, generated by $Q_k/\pi$.
\end{prop}
\begin{proof}
See \cite[Lemm. 2.4.8]{courbeFF} and \cite[Lemm. 3.2, (1)]{Ber14MultiLa}.
\end{proof}

\begin{lemm}
\label{lemma theta circ phi-n is same as mod c}
Let $x \in \At$, whose image modulo $\pi$ is $\overline{x} = (x_n)_{n \geq 0}$ in $\Et$, and assume that there exists $n \geq 0$ such that $x \in \At^{\dagger,r_n}$, so that $\overline{x} \in \Etplus$. Then for $m > n$, $\theta \circ \phi_q^{-m}(x) = x_m$ in $\O_{\Cp}/\mathfrak{a}_{\Cp}^c$ where $c = \frac{q-1}{qe}$. 
\end{lemm}
\begin{proof}
If $x \in \At^{\dagger,r_n}$, $x = \sum_{k \geq 0}\pi^k[x_k]$ in $\At$, then $\theta \circ \phi_q^{-m}(x)$ is well defined for $m \geq n$ and given by $\theta \circ \phi_q^{-m}(x)= \sum_{k \geq 0}\pi^kx_k^{(m)}$ (this is a direct consequence of lemma 5.18 of \cite{colmez2008espaces}). But then the fact that $x \in \At^{\dagger,r_n}$ implies that for $m > n$, the $\pi^kx_k^{(m)}$, $k \neq 0$ have $p$-adic valuation $\geq \frac{q-1}{qe}$. Thus $\theta \circ \phi_q^{-m}(x)=x_0^{(m)} \mod \mathfrak{a}_{\Cp}^c$.  
\end{proof}

\begin{lemm}
\label{lemma phi(x)/x means étale}
Let $x \in \Bt_{\rig}^{\dagger}$ be such that $x$ divides $\phi_q(x)$ in $\Bt_{\rig}^\dagger$ and $\frac{\phi_q(x)}{x} \in (\At^\dagger)^\times$. Then $x \in K^\times\cdot(\At^\dagger)^\times$.
\end{lemm}
\begin{proof}
Let $\tilde{D}$ denote the $\phi_q$-module over $\Bt_{\rig}^\dagger$ with base $e$ such that $\phi_q(e) = \frac{\phi_q(x)}{x} \cdot e$. By assumption, this is an étale $\phi_q$-module, so that by proposition 6.3.2 of \cite{slopes}, there exists $y \in \mathrm{GL}_1(\Bt_{\rig}^\dagger) = \Bt^\dagger \setminus \{0\}$ such that $\phi_q(y)^{-1}\frac{\phi_q(x)}{x} y = 1$. Without any loss of generality, we can assume of to replacing $y$ by $p^ky$ for some $k \in \Z$ that $y \in \At^\dagger$. But then $\frac{x}{y} \in (\Bt_{\rig}^\dagger)^{\phi_q=1} = K$ and so $y$ is equal to $x$ up to a constant, which proves the claim.
\end{proof}

\subsection{Locally analytic vectors in those rings and a conjecture of Kedlaya}
We now explain the relations between the classical point of view of locally analytic vectors in Banach representations of $p$-adic Lie groups and the new point of view of locally analytic vectors in mixed characteristic, in the context of the ring $\At^\dagger$.

In the rest of this subsection, we let $K_\infty/K$ be an infinitely ramified $p$-adic Lie extension, with Galois group $\Gamma_K$, a $p$-adic Lie group of rank $d$. We also choose coordinates $\c$ along with a nice fundamental system $(\Gamma_n)_{n \geq 1}$ of open neighborhoods of the identity of $\Gamma_K$ as in \S 2. If $R$ is a ring endowed with an action of $\G_K$ we write $R_K$ for $R^{H_K}$. 

Note that if $\rho' \leq \rho$, then $v_{\rho'} \geq v_{\rho}$ by definition. Therefore, if $x \in \At_K^{(0,\rho]}$ is such that it is $\lambda,\mu$-analytic for $\Gamma_m$, then it is also $\lambda,\mu$-analytic for $\Gamma_m$ as an element of $\At_K^{(0,\rho']}$ for all $\rho' \leq \rho$. It therefore makes sense to define $(\At_K^\dagger)^{\Gamma_m-\an,\lambda,\mu} = \varinjlim\limits_{\rho > 0}(\At_K^{(0,\rho]})^{\Gamma_m-\an,\lambda,\mu}$, and we also define $(\At_K^{\dagger})^{\Gamma_m-\an,\lambda}$ and $(\At_K^{\dagger})^{\Gamma_K-\la}$ in the same way.

\begin{lemm}
\label{lemm phi shift locana}
We have $x \in (\At_K^{(0,\rho]})^{\Gamma_K-\an, \lambda}$ if and only if $\phi_q(x) \in (\At_K^{(0,q^{-1}\rho]})^{\Gamma_K-\an, h+\lambda}$. 
\end{lemm}
\begin{proof}
This just follows from the fact that $v_{q^{-1}\rho}(\phi_q(x)) = qv_\rho(x)$ (which is item (v) of \cite[Prop. 5.4]{colmez2008espaces}).
\end{proof}

\begin{lemm}
\label{lemm locana is the same}
Let $x \in \At_K^{(0,\rho]}$. Then $x \in (\At_K^{(0,\rho]})^{\Gamma_K-\la}$ if and only if $x \in (\Bt^{[r,r]})^{\Gamma_K-\la}$, where $r = r_0/\rho$.
\end{lemm}
\begin{proof}
Let $x \in (\At_K^{(0,\rho]})^{\Gamma_K-\la}$. By definition, there exists $\lambda \in \R$  such that $x \in (\At_K^{(0,\rho]})^{\Gamma_K,\lambda-\an}$. The fact that for $r = r_0/\rho$ we have an inclusion $\At_K^{(0,\rho]} \subset \Bt^{[r,r]}$ shows that $x \in (\Bt^{[r,r]})^{\Gamma_K,\lambda-\an}$ for $v_{\rho}$. Since $v_\rho = \frac{r}{r_0}V(\cdot,r)$, this means that $x \in (\Bt^{[r,r]})^{\Gamma_m-\an,\lambda'}$, where $\lambda' = \lambda - \alpha$ with $\alpha$ such that $p^{\alpha} = \frac{r}{r_0}$, and is thus locally analytic as an element of $(\Bt^{[r,r]})$ by corollary \ref{coro Amice precise}. 

For the converse, the reasoning is the same: by corollary \ref{coro Amice precise}, if $x \in \At_K^{(0,\rho]}$ belongs to $(\Bt^{[r,r]})^{\Gamma_K-\la}$ then there exist $\lambda \in \R$ such that $x \in (\Bt_K^{[r,r]})^{\Gamma_K,\lambda-\an}$. The relation $v_\rho = \frac{r}{r_0}V(\cdot,r)$ implies that $x \in (\At_K^{(0,\rho]})^{\Gamma_m-an,\lambda'}$ where $\lambda'=\lambda +\alpha$ and so we are done. 
\end{proof}

\begin{prop}
\label{prop la in At = pa in Btrig}
Let $x \in \At_K^{\dagger}$. Then $x \in (\At_K^{\dagger})^{\Gamma_K-\la}$ if and only if $x \in (\Bt_{\rig,K}^\dagger)^{\Gamma_K-\pa}$. 
\end{prop}
\begin{proof}
Let $\rho > 0$, $m \geq 0$, and $\lambda,\mu \in \R$ be such that $x \in (\At_K^{(0,\rho]})^{\Gamma_m-\an,\lambda,\mu}$. Let $r=r_0/\rho$. By the remark above, if $s \geq r$ and $\rho'=r_0/s$, then $x \in (\At_K^{(0,\rho']})^{\Gamma_m-\an,\lambda,\mu}$, so that by lemma \ref{lemm locana is the same}, there exist $m' \geq m$, $\lambda'$, $\mu' \in \R$ such that $x \in (\Bt_K^{[s,s]})^{\Gamma_m-\an, \lambda',\mu'}$. Using the maximum principle (see corollary 2.20 of \cite{Ber02}), this implies that $x \in (\Bt_K^{[r,s]})^{\Gamma_m'-\an, \lambda'',\mu''}$ for $\lambda'' = \max(\lambda,\lambda')$ and $\mu'' = \max(\mu',\mu'')$. Therefore, $x$ belongs to $(\Bt_K^{[r,s]})^{\Gamma_K-\la}$ by corollary \ref{coro Amice precise}. Since this is true for every $s \geq r$, we deduce that $x \in (\Bt_{\rig,K}^{\dagger,r})^{\Gamma_K-\pa}$. 

For the converse, assume that $x \in (\Bt_{\rig,K}^\dagger)^{\Gamma_K-\pa}$. Then $x \in (\Bt_K^{[r,r]})^{\Gamma_K-\la}$ for any $r >0$ such that $x \in \Bt^{\dagger,r}$. Therefore, $x \in (\At_K^{(0,\rho]})^{\Gamma_K-\la}$ for $\rho = r_0/r$ by lemma \ref{lemm locana is the same} and thus $x \in (\At_K^{\dagger})^{\Gamma_K-\la}$. 
\end{proof}

\begin{coro}
We have $(\At_K^{\dagger})^{\Gamma_K-\la} = \At^{\dagger} \cap (\Bt_{\rig,K}^\dagger)^{\Gamma_K-\pa}$. 
\end{coro}

\begin{rema}
Note that since the valuations on $\At_K^{(0,\rho]}$ and $\Bt_K^{[r,r]}$ are not normalized in the same way, we do not have that $(\At_K^{\dagger})^{\Gamma_K-\la} = \At^{\dagger} \cap (\Bt_{\rig,K}^\dagger)^{\Gamma_K-\la}$. Actually, one can show that in the cyclotomic case, the ring $(\Bt_{\rig,K}^\dagger)^{\Gamma_K-\la}$ is quite small (see \S 7 of \cite{poyeton2022locally}) and does not contain $u= [\epsilon]-1$, which clearly belongs to $(\At_K^{\dagger})^{\Gamma_K-\la}$. 
\end{rema}

We now recall the conjecture of Kedlaya \cite[Conjecture 12.13]{kedlaya2013conj}. 

\begin{conj}[Kedlaya]
Let $T$ be a finite free $\Zp$-module equipped with a continuous action of $\G_K$. For $r > 0$, let $\tilde{\D}_K^{\dagger,r}(T) = (\At^{\dagger,r} \otimes_{\Zp}T)^{H_K}$ and let $\tilde{\D}_K^{\dagger,r}(T)^{\Gamma_K-\la}$ denote the set of locally analytic elements of $\tilde{\D}_K^{\dagger,r}(T)$ for the action of $\Gamma_K$ as defined in \S 2, which is a module over $(\At_K^{\dagger,r})^{\Gamma_K-\la}$. Then for any $r > 0$, the natural map
$$\At_K^{\dagger,r} \otimes_{(\At_K^{\dagger,r})^{\Gamma_K-\la}} \tilde{\D}_K^{\dagger,r}(T)^{\Gamma_K-\la} \ra \tilde{\D}_K^{\dagger,r}(T)$$
is an isomorphism.
\end{conj}

Note that the natural map $\At^{\dagger,r} \otimes_{\At_K^{\dagger,r}}\tilde{\D}_K^{\dagger,r}(T) \ra \At^{\dagger,r} \otimes_{\Zp}T$ is an isomorphism thanks to for example \S 8 of \cite{KLrelative}. We quickly remark that thanks to lemma \ref{lemm locana is the same} it is easy to check that the definitions of locally analytic elements in $\At^{\dagger,r}$ used in \cite{kedlaya2013conj} coincide with ours. 

We finish this section by pointing out that, as soon as $K_\infty$ contains a twisted cyclotomic extension, then Kedlaya's conjecture holds for this extension by theorem 9.1 of \cite{Ber14MultiLa}.

\section{Complements on holomorphic functions}
Let $\cal{A}_K^{\dagger,r}$ denote the set of Laurent series $\sum_{k \in \Z}a_kT^k$ with coefficients in $\O_K$ such that $v_p(a_k)+kr/e \geq 0$ for all $k \in \Z$ and such that $v_p(a_k)+kr/e \rightarrow +\infty$ when $k \rightarrow -\infty$. We endow $\cal{A}_K^{\dagger,r}$ with a valuation $v_r$ given by $v_r(f(T)) = \inf_{k \in \Z}(v_p(a_k)+kr/e)$ if $f(T) = \sum_{k \in \Z}a_kT^k$. We let $|\cdot|_r$ denote the norm on $\cal{A}_K^{\dagger,r}$ defined by $|x|_r = p^{-v_r(x)}$. We let $\cal{B}_K^{\dagger,r} = \cal{A}_K^{\dagger,r}[1/p]$. We also let $C([r,+\infty[) := \{z \in \Cp, p^{-1/r} \leq |z|_p < 1\}$.

We let $\cal{R}_K^s$ denote the Fréchet completion of $\cal{A}_K^{\dagger,s}[1/p]$ for the valuations $v_{s'}$, $s' \geq s$, and we let $\cal{R}_K = \bigcup_{r > 0}\cal{R}_K^r$ and $\cal{B}_K^\dagger = \bigcup_{r > 0}\cal{B}_K^{\dagger,r}$. 

Most of the results mentioned here regarding the rings $\cal{B}_K^{\dagger,s}$ are well known and the proofs can be found in \cite{lazardfonctions1962}. We provide a proof or a reference of the other statements.

\begin{prop}
\label{prop Lazard PID}
The ring $\cal{B}_K^{\dagger,r}$ is a PID, whose ideals are generated by elements of $K[T]$ having no roots in $C([r,+\infty[)$.
\end{prop}

\begin{lemm}
We have $(\cal{R}_K)^\times = \cal{B}_K^{\dagger} \setminus \{0\}$, and an element of $\cal{B}_K^{\dagger,r}$ is invertible (in $\cal{B}_K^{\dagger,r}$) if and only if it has no zeroes in $C([r,+\infty[)$. 
\end{lemm}

\begin{lemm}
If $h \in \cal{R}_K^r$, then $h \in \cal{B}_K^{\dagger,r}$ if and only if it has finitely many zeroes in $C([r,+\infty[)$, if and only if the function $s \mapsto |h|_s$ is bounded as $s \rightarrow +\infty$. 
\end{lemm}

\begin{lemm}
\label{lemma f circ f divisible}
Let $f(T) \in T\cdot (\cal{A}_K^{\dagger,r})^\times$ be such that $f(T)-T$ has finite Weierstrass degree $d$. Then for any $n \geq 1$, $f^{\circ n}(T) - T$ is divisible by $f(T)-T$ in $\cal{A}_K^{\dagger,r}$. 
\end{lemm}
\begin{proof}
Let $x$ be a root in $\mathfrak{\Cp}$ of $f(T)-T$. Then $f(x) = x$ and thus $f^{\circ n}(x) =x$ so that the roots of $f(T)-T$ are also roots of $f^{\circ n}(T) - T$. Assume that $x$ is a double root of $f(T)-T$, so that $f'(x) = 1$. Then $(f^{\circ n}(T) - T)'(x) = f'(x)\cdot (f^{\circ (n-1)})'(f(x))-1$ and since $f(x)=x$ and $f'(x)=1$, we also have that $x$ is a double root of $f^{\circ n}(T) - T$ for $n=2$, and then for any $n$ by induction. Repeating the same argument for the higher order derivatives shows that each root of $f(T)-T$ is also a root of $f \circ f(T) - T$ of at least the same multiplicity.

This means that $f(T)-T$ divides $f \circ f(T) - T$ in $\cal{R}_K^r$ (the meromorphic function $h(T):=(f^{\circ n}(T) - T)/(f(T)-T)$ has no zeroes on the corresponding annulus so that it is holomorphic and thus belongs to $\cal{R}_K^r$). Since $f \circ f(T) - T$ is bounded, it has only a finite number of zeroes and thus so does $h$. Therefore, $h(T) \in \cal{A}_K^{\dagger,r}[1/\pi]$. 

Let us write 
\[f^{\circ n}(T) - T = f^{\circ n}(T)-f^{\circ (n-1)}+f^{\circ (n-1)}- \cdots -T.\]

Each $f^{\circ k}(T)-f^{\circ (k-1)}(T)$ can be written as $(f(T)-T)\circ f^{\circ (k-1)}$, and since $f(T) \in T\cdot (\cal{A}_K^{\dagger,r})^\times$, we have for any $s \geq r$, $|f^{\circ k}(T)-f^{\circ (k-1)}(T)|_s = |f(T)-T|_s$. This implies that for any $s \geq r$, $|h(T)|_s \leq 1$. Writing $h(T) = \sum_{n \in \Z}a_nT^n$, this means that $|a_n|\rho^n \leq 1$ for all $n \in \Z$ and $\rho < 1$ close to $1$, so that $a_n \in \O_K$ for all $n$, and so $h(T) \in \cal{A}_K^{\dagger,r}$.
\end{proof}

\begin{lemm}
If $x \in \Bt^I$ with $I=[r,s]$, then $V(I,x) = \inf(V(r,x),V(s,x))$. 
\end{lemm}
\begin{proof}
This is corollary 2.20 of \cite{Ber02}. 
\end{proof}

\begin{coro}
\label{coro function bounded or eventually increasing}
If $x \in \Bt_{\rig}^\dagger$, then either there exists $r > 0$ such that $V([r,+\infty[,x) = V(r,x)$, or the function $r \mapsto V(r,x)$ is eventually decreasing as $r \rightarrow +\infty$.
\end{coro}

\begin{lemm}
\label{lemm when is function bounded}
Let $x \in \Bt_{\rig}^\dagger$. Then $x \in \Bt^\dagger$ if and only if the function $r \mapsto V(r,x)$ is bounded as $r \rightarrow +\infty$. Moreover, $(\Bt_{\rig}^\dagger)^\times = \Bt^{\dagger} \setminus \{0\}$. 
\end{lemm}
\begin{proof}
The first part is a direct consequence of our definition of $\Bt^\dagger$, along with proposition 1.6.23 of \cite{courbeFF}. For the statement on the invertible elements, this is proposition 1.8.6 of ibid. 
\end{proof}

\section{Overconvergent lifts of the field of norms}
\label{section OC lift}
Let $K_\infty$ be an infinite totally ramified Galois extension of $K$ whose Galois group is a $p$-adic Lie group. The main theorem of \cite{sen1972ramification} shows that $K_\infty/K$ is ``strictly arithmetically profinite'' (or strictly APF) in the terminology of \cite{Win83} and we can thus apply the field of norms construction of ibid. to $K_\infty/K$. We let $X_K(K_\infty)$ be the field of norms attached to the extension $K_\infty/K$, which is a local field of characteristic $p$ with residue field $k_K$ by theorem 2.1.3 of ibid. In particular there exists a uniformizer $u$ of $X_K(K_\infty)$ such that $X_K(K_\infty) = k_K(\!(u)\!)$. Moreover, this field comes equipped with an action of $\Gamma_K$ and of the absolute Frobenius $\phi : x \mapsto x^p$.

If we let $\cal{E}_{K_\infty}$ denote the set of finite subextensions $K \subset E \subset K_\infty$, then by definition, elements of $X_K(K_\infty)$ are norm-compatible sequences $(x_E)_{E \in \cal{E}(K_\infty)}$ such that $x_E \in E$ for all $E \in \cal{E}(K_\infty)$, and $N_{F/E}(x_F) = x_E$ whenever $E,F \in \cal{E}(K_\infty)$, $E \subset F$.

Since $K_\infty/K$ is strictly APF, there exists by \cite[4.2.2.1]{Win83} a constant $c = c(K_{\infty}/K) > 0$ such that for all $F \subset F'$ finite subextensions of $K_\infty/K$, and for all $x \in \mathcal{O}_{F'}$, we have 
$$v_K(\frac{N_{F'/F}(x)}{x^{[F':F]}}-1) \geq c.$$

We can always assume that $c \leq v_K(p)/(p-1)$ and we do so in what follows. By \S 2.1 and \S 4.2 of \cite{Win83}, there is a canonical $\G_K$-equivariant embedding $\iota_K : A_K(K_\infty) \hookrightarrow \Etplus$, where $A_K(K_\infty)$ is the ring of integers of $X_K(K_\infty)$. We can extend this embedding into a $\G_K$-equivariant embedding $X_K(K_\infty) \hookrightarrow \Et$ where $\Et$ is the fraction field of $\Etplus$, and we note $\E_K$ its image. We also let $\E_K^+$ denote the ring of valuation of $\Et_K$. We can actually give an explicit description of this embedding.

\begin{prop}
\label{prop FON embedding}
Let $0 < c \leq c(K_\infty/K)$.
\begin{enumerate}
\item the map $\iota_K : A_K(K_\infty) \ra \varprojlim_{x \mapsto x^q} \O_{K_\infty} / \mathfrak{a}^c_{K_\infty} = \Et^+_K$ is injective and isometric;
\item the image of $\iota_K$ is $\varprojlim_{x \mapsto x^q} \O_{K_n} / \mathfrak{a}^c_{K_n}$.
\end{enumerate}
\end{prop}
\begin{proof}
This is proven in \S 4.2 of \cite{Win83}.
\end{proof}

Let $E$ be a finite extension of $\Qp$, with residue field $k_E=k_K$. Let $\varpi_E$ be a uniformizer of $E$, and let $\A_K$ denote the $\varpi_E$-adic completion of $\mathcal{O}_E[\![T]\!][1/T]$ (the notation $\A_K$ is used here for compatibility with the action of $\G_K$ but be mindful that this is actually dependent on $E$ even if it does not appear in the notation). The ring $\A_K$ is a $\varpi_E$-Cohen ring of $X_K(K_\infty)=k_K(\!(\pi_K)\!)$, and following the definition of \cite{Ber13lifting}, we say that the action of $\Gamma_K$ is liftable if there exists such a field $E$ and power series $\{F_g(T)\}_{g \in \Gamma_K}$ and $P(T)$ in $\A_K$ such that:
\begin{enumerate}
\item $\overline{F}_g(\pi_K) = g(\pi_K)$ and $\overline{P}(\pi_K) = \pi_K^q$;
\item $F_g \circ P = P \circ F_g$ and $F_g \circ F_h = F_{hg}$ for all $g,h \in \Gamma_K$;
\end{enumerate}
where the notations $\overline{F}_g$ and $\overline{P}$ stand for the reduction of the power series mod $\varpi_E$.

When the action of $\Gamma_K$ is liftable we get a $(\phi,\Gamma)$-module theory as in Fontaine's classical cyclotomic theory \cite{Fon90} in order to study $\mathcal{O}_E$-representations of $\G_K$, replacing the cyclotomic extension in the theory of Fontaine by the extension $K_\infty/K$. In particular, if the action of $\Gamma_K$ is liftable, then there is an equivalence of categories between étale $(\phi_q,\Gamma_K)$-modules on $\A_K$ and $\mathcal{O}_E$-linear representations of $\G_K$ (see \cite[Thm. 2.1]{Ber13lifting}).

\begin{prop}
\label{uembedding}
There is a $\G_K$-equivariant embedding $\A_K \hookrightarrow \At_K$ and compatible with $\phi_q$ that lifts the embedding $\iota_K: X_K(K_\infty) \hookrightarrow \Et_K=\Et^{\Gal(\Qpbar/K_\infty)}$.
\end{prop}
\begin{proof}
See \cite[A.1.3]{Fon90} or \cite[\S 3]{Ber13lifting}.
\end{proof}

We say that a lift of the field of norms is overconvergent if the power series $P(T)$ giving the lift of the Frobenius belongs to $\cal{A}_K^{\dagger,r}$ for some $r > 0$. 

We now assume that there is an overconvergent lift of the field of norms. Let $u \in \At_K$ be the image of $T$ by the embedding given by proposition \ref{uembedding}, so that $\phi_q(u) = P(u)$ and $g(u) = F_g(u)$ for $g \in \Gamma_K$. 

\begin{lemm}
\label{lemm liftsurconv implies unif surconv}
If $P(T) \in \cal{A}_K^{\dagger,r}$, with $r \geq \frac{p-1}{pe}$ then $u \in \At^{\dagger,r}$.
\end{lemm}
\begin{proof}
We have $u \in \At$, and we write $u = (u-[\overline{u}])+[\overline{u}]$. Since $\overline{u} \in \Etplus$, we have $u \in (\Atplus+\varpi_E\At)$. Let us write $P(T) = P^+(T)+P^{-}(1/T)$, with $P^+(T) \in T^q+\mathfrak{m}_E[\![T]\!]$ and $P^-(T) = \sum_{n > 0}a_nT^n \in \mathfrak{m}_E[\![T]\!]$ with $v_p(a_n) \geq ne/r$. We thus have
\[ P^+(u) \in (P^+([\overline{u}])+\varpi_E^2\At) \subset (\Atplus+\varpi_E^2\At)\]
and
\[ P^-\left(\frac{1}{u}\right)= P^-\left(\frac{1}{[\overline{u}]}\frac{1}{1+\frac{u-[\overline{u}]}{[\overline{u}]}}\right) \in (P^-(\frac{1}{[\overline{u}]})+\varpi_E^2\At)\]
since $\frac{1}{1+\frac{u-[\overline{u}]}{[\overline{u}]}} \in 1+\varpi_E\At$. Thus $P(u) \in (P([\overline{u}])+\varpi_E^2\At) \subset (\At^{\dagger,r}+\varpi_E^2\At)$.

Therefore, $u = \phi_q^{-1}(P(u)) \in (\At^{\dagger,r/q}+\varpi_E^2\At) \subset (\At^{\dagger,r}+\varpi_E^2\At)$. Now let us assume that $u \in (\At^{\dagger,r}+\varpi_E^k\At)$ for some $k \geq 2$. Let us write $u = a+b$, with $a \in \At^{\dagger,r}$ and $b \in \varpi_E^k \At$. Since $\overline{u} = \overline{a}$, we have that $\frac{a}{[\overline{a}]}$ belongs to and is a unit of $\At^{\dagger,r'}$ with $r' = r+\frac{p-1}{pe}$ by lemma \ref{lemma v in r not just rho}. Writing $a = \frac{a}{[\overline{a}]}[\overline{a}]$ shows that $b/a \in \varpi_E^k\At$ and thus
\[ \frac{1}{u} = \frac{1}{a}\left(\frac{1}{1+\frac{b}{a}}\right) \in (\frac{[\overline{a}]}{a}\frac{1}{[\overline{a}]}+\varpi_E^k\At).\]
Therefore, we have
\[P(u) \in P(\frac{[\overline{a}]}{a}\frac{1}{[\overline{a}]})+\varpi_E^{k+1}\At) \subset \At^{\dagger,r'}+\varpi_E^{k+1}\At)\]
and thus $u = \phi_q^{-1}(P(u)) \in (\At^{\dagger,r'/q}+\varpi_E^{k+1}\At) \subset (\At^{\dagger,r}+\varpi_E^{k+1}\At)$. We can now conclude by using the fact that for any $r > 0$, we have $\bigcap_{k \geq 0}(\At^{\dagger,r}+\varpi_E^k\At) = \At^{\dagger,r}$ (which follows from the definition of $\At^{\dagger,r}$). 
\end{proof}

\begin{rema}
If one looks closely at the proof of lemma \ref{lemm liftsurconv implies unif surconv}, one could improve the radius of overconvergence of $u$, but we don't need this level of precision here.
\end{rema}

By \cite[Rem. 4.3]{Ber13lifting}, we have the following:

\begin{prop}
\label{prop liftsurconv implies reg of Galois action}
For all $g \in \Gamma_K$, $F_g(T) \in T\cdot(\cal{A}_K^{\dagger,r})^\times$.
\end{prop}
\begin{proof}
This is basically the same proof as the one of proposition 4.2 of \cite{Ber13lifting}, except that there is a small gap in the proof which is fixed in \cite{erratumBerger}. 

The ring $\A_K$ is a free $\phi_q(\A_K)$-module of rank $q$, and we define $\cal{N}: \A_K \ra \A_K$ to be the map
\[ \cal{N} : f(T) \mapsto \phi_q^{-1} \circ N_{\A_K/\phi_q(\A_K)}(f(T)).\]
Since $\A_K^\dagger$ is a free $\phi_q(\A_K^\dagger)$-module of rank $q$, we have $\cal{N}(\A_K^\dagger) \subset \A_K^\dagger$. 

By construction, $v_r(\cal{N}(T)) = v_r(T)$, and $\cal{N}(T)$ is equal to $T$ modulo $\mathfrak{m}_E$, so that $\cal{N}(T) \in T\cdot(\cal{A}_K^{\dagger,r})^\times$. Let $\cal{N}'$ denote the map 
\[ \cal{N}' : f(T) \mapsto (\cal{N}(T))^{-1}\cal{N}(f(T)),\]
so that $\cal{N}'(T) = T$. 

Now for $k \geq 1$, one has $\cal{N}'(T\cdot(\cal{A}_K^{\dagger,r})^\times+\varpi_E^k\A_K) \subset T\cdot(\cal{A}_K^{\dagger,r})^\times+\varpi_E^{k+1}\A_K$ (see proposition 2.3.2 of \cite{Fon90}) and so by induction on $k$, this implies that 
\[(T\cdot(\cal{A}_K^{\dagger,r})^\times+\varpi_E\A_K)^{\cal{N}'(x)=x} \subset T\cdot(\cal{A}_K^{\dagger,r})^\times .\]

Since $F_g(T) \in T\cdot(\cal{A}_K^{\dagger,r})^\times+\varpi_E\A_K$ and $\cal{N}'(g(T)) = g(T)$ if $g \in \Gamma_K$, we obtain that $F_g(T) \in (T\cdot(\cal{A}_K^{\dagger,r})^\times+\varpi_E\A_K)^{\cal{N}'(x)=x} \subset T\cdot(\cal{A}_K^{\dagger,r})^\times.$
\end{proof}

\begin{lemm}
\label{lemma coeffs mahler h(T) in terms of T}
For $g \in \Gamma_n$, let $\Delta_g : \cal{A}_K^{\dagger,r} \ra \cal{A}_K$ be the map defined by $h(T) \mapsto h(F_g(T))-h(T)$. We have $|\!|\Delta_g(h(T))|\!|_r \leq |\!|T-F_g(T)|\!|_r|\!|h(T)|\!|_r$ and so in particular the target of the map is contained in $\cal{A}_K^{\dagger,r}$.
\end{lemm}
\begin{proof}
In order to prove this claim, we write $h = h^++h^-$, where $h(T) = \sum_{n \in \Z}a_nT^n$, $h^+(T) = \sum_{n \geq 0}a_nT^n$ and $h^-(T) = \sum_{n < 0}a_nT^n$, which we rewrite as $h^-(T) = \sum_{n > 0}b_nT^{-n}$. Then
\[h^+(F_g(T))-h^+(T) = \sum_{n \geq 0}a_n\left((F_g(T)^n-T^n\right)= \sum_{n > 0}a_n(F_g(T)-T)(\sum_{k=0}^nF_g(T)^kT^{n-k})\]
and since $|\!|F_g(T)|\!|_r = |\!|T|\!|_r$ by proposition \ref{prop liftsurconv implies reg of Galois action}, this means that 

\[|\Delta_g(h^+(T))|_r \leq \sum_{n > 0}|a_n|_p|\!|\Delta_g(T)|\!|_r|\!|h^+|\!|_r. \]

We do the same for $h^-$: we have $h^-(F_g(T))-h^-(T) = \sum_{n \geq 1}b_n(\frac{1}{F_g(T)^n}-\frac{1}{T^n})$. We write $B(T) = \frac{T}{F_g(T)} \in (\cal{A}_K^{\dagger,r})^\times$. Thus 
\[ \frac{1}{F_g(T)^n}-\frac{1}{T^n} = \frac{T^n-F_g(T)^n}{(TF_g(T))^n} = \frac{T^n-F_g(T)^n}{T^{2n}}B(T)^n.\]
and hence 
\[ \frac{1}{F_g(T)^n}-\frac{1}{T^n} = (T-F_g(T))\sum_{k=0}^n\left(\frac{F_g(T)}{T}\right)^k\frac{B(T)^n}{T^n}.\]
Since $\frac{F_g(T)}{T}$ is a unit of $\cal{A}_K^{\dagger,r}$ (and so is $B(T)$), we obtain that 
\[|\!|(T-F_g(T))\sum_{k=0}^n\left(\frac{F_g(T)}{T}\right)^k\frac{B(T)^n}{T^n}|\!|_r = |\!|T-F_g(T)|\!|_r |\!|T^{-n}|\!|_r\]
so that 
\[|\!|\Delta_g(h^-(T))|\!|_r \leq |\!|T-F_g(T)|\!|_r \sum_{n > 0}|b_n|_p|\!|T^{-n}|\!|_r = |\!|T-F_g(T)|\!|_r|\!|h^-(T)|\!|_r.\]

Therefore, $|\!|\Delta_g(h(T))|\!|_r \leq |\!|T-F_g(T)|\!|_r|\!|h(T)|\!|_r$.
\end{proof}

\begin{prop}
\label{prop liftsurconv implies locana}
We have $u \in (\At_K^\dagger)^{\Gamma_K-\la}$.
\end{prop}
\begin{proof}
By lemma \ref{lemma coeffs mahler h(T) in terms of T}, we obtain that  $|\!|\Delta_g^n(u))|\!|_r \leq |\!|\Delta_g(u)|\!|_r^n$ and we can now apply proposition 2.3 of \cite{porat2024locally}, which shows that $u$ is locally analytic for the action of $\Gamma_K$.
\end{proof}

\section{Structure of locally analytic vectors in $\At^\dagger$ for $\Z_p$-extensions}
\label{section struc locana Zp}
In this section, we assume that $K_\infty/K$ is a totally ramified $\Zp$-extension, with Galois group $\Gamma_K \simeq \Zp$. The goal of this section is to prove that if there are nontrivial locally analytic vectors in $\At_K^\dagger$, that is if $(\At_K^\dagger)^{\Gamma_K-\la} \neq \O_{K}$, then everything behaves just as if $K_\infty/K$ was the cyclotomic extension. 

Let $K_\infty/K$ be a totally ramified $\Zp$-extension, with Galois group $\Gamma_K \simeq \Zp$. We assume furthermore that $(\At_K^\dagger)^{\la} \neq \O_{K}$, which means that it contains a nontrivial locally analytic vector. For $n \geq 1$ we let $K_n/K$ be the subextension of $K_\infty/K$ such that $\Gal(K_n/K) = \Z/p^n\Z$ and we let $\Gamma_n = \Gal(K_\infty/K_n) \subset \Gamma_K$. We also let $H_K = \Gal(\overline{K}/K_\infty)$. Note that, up to extending the field $K$, we can always assume without loss of generality that $K/\Qp$ is Galois, and we do so in what follows.

We let $s : \At \ra \At/\pi\At \simeq \Et$ denote the projection map given by the reduction modulo $\pi$. Note that it induces by restriction projections that we will still denote by $s$: $\At_K \ra \Et_K$, $\At^\dagger \ra \Et$ and $\At_K^\dagger \ra \Et_K$ and whose kernel is still generated by $\pi$. 

\begin{prop}
\label{prop exists lambda included in FON}
We have $(\Et_K)^{\Gamma_K,0-\an} \subset \E_K$, and we have $(\Et_K)^{\Gamma_K-\la} = \bigcup_{n \geq 0}\phi_q^{-n}(\E_K)$.
\end{prop}
\begin{proof}
This is theorem 2.2.3 of \cite{berger2022super} since $\dim \Gamma_K = 1$. 
\end{proof}

In what follows, we choose the smallest integer $\lambda$ such that $(\Et_K)^{\Gamma_K,\lambda-\an} \subset \E_K$. In particular, $\lambda \leq 0$.

\begin{coro}
\label{lemma image mod p of AtKdaggerla}
We have $s((\At_K^\dagger)^{\Gamma_K,\lambda-\an}) \subset \E_K$ and $s((\At_K^\dagger)^{\Gamma_K-\la}) \subset \phi_q^{-\infty}(\E_K)$.
\end{coro}
\begin{proof}
Let $x \in (\At_K^\dagger)^{\Gamma_K,\lambda-\an}$. Then $s(x) \in \At_K^\dagger/\pi\At_K^\dagger \simeq \Et_K$ is $\lambda$-analytic (for $\Gamma_K$) for the valuation induced on $\Et_K$ by the one on $\At_K^\dagger$. Proposition \ref{prop exists lambda included in FON} shows that $s(x) \in \E_K$, so this proves the first part of the corollary. The second part comes from the fact that $x \in \At^{(0,\rho]}$ is $\kappa$-analytic (for $\Gamma_K$) if and only if $\phi_q^\ell(x)$ is $(\kappa-f\ell$)-analytic (for $\Gamma_K$) by lemma \ref{lemm phi shift locana}.   
\end{proof}

\begin{lemm}
\label{lemm exists nontrivial element}
There exists $x \in (\At_K^\dagger)^{\Gamma_K,\lambda-\an}$ whose image by $s$ belongs to $\E_K^+ \setminus k_K$.
\end{lemm}
\begin{proof}
Since $(\At_K^\dagger)^{\la}$ is non trivial, there exists $z \in (\At_K^\dagger)^{\la}$ not in $\O_K$. In particular, writing $z = \sum_{k \geq 0}\pi^k[z_k]$, at least one of the $z_k$ does not belong to $k_K$. Let $j$ denote the smallest such $k$ and let $\tilde{z} = \sum_{k = 0}^{j-1}\pi^k[z_k]$, so that $\tilde{z} \in \O_K$ by assumption. Then $z-\tilde{z}$ is divisible by $\pi^j$, and locally analytic, so that $y:=\frac{z-\tilde{z}}{\pi^j}$ is also locally analytic, and its image modulo $\pi$ is $z_k$ which does not belong to $k_K$. 

We can assume up to replacing $y$ by its inverse that $s(y)$ has nonnegative valuation so that it belongs to $\Etplus$: since $s(y) \neq 0$, the inverse of $y$ belongs to $\At^\dagger$ and not just $\Bt^\dagger$, and its inverse is locally analytic by lemma \ref{lemma invertible loc ana}. 

By lemma \ref{lemma Gulotta cangobackto GammaK}, there exists $\kappa \in \R$ such that $y \in (\At_K^\dagger)^{\Gamma_K,\kappa-\an}$. By applying $\phi_q^\ell$ to this element for $\ell \gg 0$, we find using lemma \ref{lemm phi shift locana} that there exists $x \in (\At_K^\dagger)^{\Gamma_K,\lambda-\an}$ whose image by $s$ is an element of $\E_K \setminus k_K$, and since $y \in \Etplus$ then $s(x) = \phi_q^{\ell}(s(y))$ belongs to $\E_K \cap \Etplus = \E_K^+$. 
\end{proof}

\begin{defi}
\label{def of lift of unif of A_K}
We let $\alpha := \min\{v_{\E}(s(x)): x \in (\At_K^\dagger)^{\Gamma_K,\lambda-\an} \textrm{ and } v_{\E}(s(x)) > 0\}$. 

Note that the set of elements $x$ in $(\At_K^\dagger)^{\Gamma_K,\lambda-\an}$ such that the valuation of $s(x)$ is nonzero is nonempty by lemma \ref{lemm exists nontrivial element}, and that the set of $s(x)$ for $x \in (\At_K^\dagger)^{\Gamma_K,\lambda-\an}$ is included in $\E_K^+$ by corollary \ref{lemma image mod p of AtKdaggerla}. Since the valuation on $\E_K^+$ is discrete, this means that $\alpha$ is well defined, and that the minimum is reached for some element in $(\At_K^\dagger)^{\Gamma_K,\lambda-\an}$ which will be denoted by $v$.

Since $\alpha = v_{\E}(s(v)) > 0$, the sequence $(v^n)_{n \geq 0}$ goes to $0$ in $\At_K$ for the $(\pi,[s(v)])$-adic topology (for which $\At$ and $\At_K$ are complete), and thus $\O_{K}(\!(v)\!)$ is naturally a subring of $\At_K$. We let $\A_K$ denote the $\pi$-adic completion of $\O_{K}(\!(v)\!)$ in $\At_K$ (we recall that $\At_K$ is $\pi$-adically complete).
\end{defi}

In the definition \ref{def of lift of unif of A_K} above, we can thanks to lemma \ref{lemma v in r not just rho} assume that our choice of $v$ satisfies the additional assumption that there exists $n \geq 0$ such that $v \in \At^{\dagger,r_n}$ and such that $\frac{v}{[s(v)]}$ belongs to $\At^{\dagger,r_n}$ and is a unit of this ring. In order to avoid additional notations we write $r$ for $r_n$ in the rest of this section.

\begin{lemm}
\label{lemm s(ana) is power series in v}
We have $s((\At_K^\dagger)^{\Gamma_K,\lambda-\an}) \subset k_K(\!(s(v))\!)$.
\end{lemm}
\begin{proof}
Let $x \in (\At_K^\dagger)^{\Gamma_K,\lambda-\an}$. By corollary \ref{lemma image mod p of AtKdaggerla}, we know that $s(x) \in \E_K$. Let $k = v_{\E}(s(x))$, and let $k=q_0\alpha+r$ be the euclidean division of $k$ by $\alpha$. We have $v^{-q_0}x \in (\At_K^\dagger)^{\Gamma_K,\lambda-\an}$ by lemma \ref{lemma lambda an ring}, and $0 \leq v_{\E}(s(v^{-q_0}x)) = r < \alpha$, so that $r=0$ by definition of $\alpha$. There exists therefore $c_0 \in k_K$ such that $z:=[c_0]v^{q_0}$ satisfies $v_{\E}(s(x)-s(z)) > v_{\E}(s(x))$. Now $x-z \in (\At_K^\dagger)^{\Gamma_K,\lambda-\an}$ and thus we can apply the same reasoning to $x-z$ instead of $x$. This yields $c_1 \in k_K$ and $q_1 > q_0$ such that $z_1:=x-[c_0]v^{q_0}+[c_1]v^{q_1}$ is such that $v_{\E}(s(z_1)) > v_{\E}(s(x-z))$. Applying the same process inductively gives us $(q_i)_{i \geq 0} \in \Z^\N$ an increasing sequence and $(c_i)_{i \geq 0} \in k_K^\N$ such that $s(x) = \sum_{i=0}^{+\infty} c_is(v)^{q_i}$ and thus $s(x) \in k_K(\!(s(v))\!)$. 
\end{proof}

\begin{lemm}
\label{lemma locana inside padic compl of OK((v))}
We have $(\At_K^\dagger)^{\Gamma_K,\lambda-\an} \subset \A_K$.
\end{lemm}
\begin{proof}
Let $x \in (\At_K^\dagger)^{\Gamma_K,\lambda-\an}$. By lemma \ref{lemm s(ana) is power series in v}, we know that $s(x) \in k_K(\!(s(v))\!)$. Therefore, there exists $P_0(T) \in O_K(\!(T)\!)$ such that $x-P_0(v) \in \pi\At_K^{\dagger}$. Moreover, since $x,v \in (\At_K^\dagger)^{\Gamma_K,\lambda-\an}$, this implies that $x-P_0(v) \in (\At_K^\dagger)^{\Gamma_K,\lambda-\an} \cap \pi\At_K^\dagger = \pi(\At_K^\dagger)^{\Gamma_K,\lambda-\an}$. Let $z = \frac{x-P_0(v)}{\pi} \in (\At_K^\dagger)^{\Gamma_K,\lambda-\an}$. Then applying the same process for $z$ instead of $x$ yields $P_1(T) \in O_K(\!(T)\!)$ such that $x-P_0(v)-\pi\cdot P_1(v) \in \pi^2\At_K^{\dagger}$. Inductively, we find a sequence $(P_i(T))_{i \geq 0}$ of elements of $O_K(\!(T)\!)$ such that $x = \sum_{i=0}^{\infty}\pi^i\cdot P_i(v)$, and this series converges in $\At_K$ since it is $\pi$-adically complete, to an element of $\A_K$ by definition of $\A_K$. 
\end{proof}

\begin{lemm}
We have $(\At_K^\dagger)^{\Gamma_K-\la} \subset \phi_q^{-\infty}(\A_K)$. 
\end{lemm}
\begin{proof}
Let $x \in (\At_K^\dagger)^{\Gamma_K-\la}$. Therefore there exists $m \geq 0$ and $\mu \in \R$ such that $x \in (\At_K^\dagger)^{\Gamma_m,\mu-\an}$. Note that by lemma 1.10 of \cite{BerRozdecompletion}, this is equivalent to the existence of $\mu' \in \R$ such that $x \in (\At_K^\dagger)^{\Gamma_K,\mu'-\an}$. If $k$ is an integer such that $kh+\mu' \geq \lambda$, then $\phi_q^k(x) \in (\At_K^\dagger)^{\Gamma_K,(\mu'+kf)-\an} \subset (\At_K^\dagger)^{\Gamma_K,\lambda-\an}$ so that $\phi_q^k(x) \in \A_K$ by lemma \ref{lemma locana inside padic compl of OK((v))}, and thus $x \in \phi_q^{-k}(\A_K)$. 
\end{proof}

Recall that $\cal{A}_K^{\dagger,s}$ denotes the set of Laurent series $\sum_{k \in \Z}a_kT^k$ with coefficients in $\O_K$ such that $v_p(a_k)+ks/e \geq 0$ for all $k \in \Z$ and such that $v_p(a_k)+ks/e \rightarrow +\infty$ when $k \rightarrow -\infty$. 

Recall also that there exists some $r > 0$ such that $v \in \At^{\dagger,r}$ and such that $\frac{v}{[s(v)]}$ is a unit in $\At^{\dagger,r}$. For $s \geq r$, we let $\A_K^{\dagger,s}$ denote the set of $P(v) \in \At$ such that $P \in \cal{A}_K^{\dagger,s/\alpha}$. We also let $\A_K^\dagger = \cup_{s \geq r}(\A_K^{\dagger,s}[1/v])$. 

\begin{prop}
\label{prop locana belongs to surconvring}
For $s \geq r$, we have $(\At^{\dagger,s}_K)^{\Gamma_K,\lambda-\an} = \A_K^{\dagger,s}$. 
\end{prop}
\begin{proof} 
By lemma \ref{lemma v in r not just rho} and the choice of $r$ we made, $v$ belongs to $\At^{\dagger,r}$ and is such that $\frac{v}{[s(v)]}$ is a unit in $\At^{\dagger,r}$.

Now the proof of item (i) of \cite[Prop. 7.5]{colmez2008espaces} carries over and shows that $\A_K \cap \At^{\dagger,r}_K = \A_K^{\dagger,r}$ for $s \geq r$, using the fact that $v_{\E}(v) = \alpha$. To finish the proof, it suffices to notice that if $x \in (\At^{\dagger,r}_K)^{\Gamma_K,\lambda-\an}$, then $x \in \A_K$ by lemma \ref{lemma locana inside padic compl of OK((v))}, and $x$ belongs to $\At^{\dagger,r}_K$. 
\end{proof}

\begin{coro}
\label{coro la implies lift}
There exists $Q(T)$ in $\cal{A}_K^{\dagger,qr/\alpha}$ such that $\phi_q(v) = Q(v)$ and for each $g \in \Gamma_K$, there exists a series $F_g(T)$ in $\cal{A}_K^{\dagger,r/\alpha}$ such that $g(v) = F_g(v)$. 
\end{coro}

\begin{coro}
\label{coro Atdaggerla in phi-infAKdagger}
We have $(\At_K^\dagger)^{\Gamma_K-\la} = \phi_q^{-\infty}(\A_K^{\dagger})$. 
\end{coro}

\begin{prop}
\label{prop v actually uniformizer upto phi}
There exist $k \geq 0$ and $w \in (\At_K^\dagger)^{\Gamma_K-\la}$ such that $\phi^{-k}(s(w))$ is a uniformizer of $\E_K$. 
\end{prop}
\begin{proof}
Let $u$ be a uniformizer of $\E_K$, and let us write $\overline{v}$ for $s(v)$. Note that $k_K(\!(u)\!)/k_K(\!(\overline{v})\!)$ is a finite extension of local fields of characteristic $p$. It can thus be decomposed as a purely inseparable extension of a separable extension of $k_K(\!(\overline{v})\!)$, so that there exists $k \geq 0$ and a separable monic polynomial $P$ with coefficients in $k_K(\!(\overline{v})\!)$ such that $\phi^k(u)$ is a root of $P$. Now let $y:=\phi^k(u)$ and let $\tilde{P}(T) \in \O_K(\!(v)\!)[T] \subset \Bt_K^\dagger$ be a lift of $P$ which is monic. Since $\B_K^\dagger:=\A_K^{\dagger}[1/p]$ is a Henselian field (cf \S 2 of \cite{matsuda1995local}), and since $\Bt^\dagger$ is absolutely unramified and has $\Et$ as a residue field which contains $\E_K$, there exists $\tilde{y} \in \Bt_K^\dagger$ lifting $y$ such that $\tilde{P}(\tilde{y}) = 0$ and by construction $\tilde{y} \in \At_K^\dagger$ and $\tilde{P}'(\tilde{y}) \neq 0$. 

Since $\tilde{P}'(\tilde{y}) \neq 0$ and since $\Bt_K^\dagger$ is a field, there exists $r > 0$ such that $\tilde{P}'(\tilde{y})$ is invertible in $\Bt_K^{\dagger,r}$ and such that all the coefficients of $\tilde{P}$ belong to $\B_K^{\dagger,r} \subset \Bt_K^{\dagger,r}$ (up to increasing $r$ if needed for the last inclusion to make sense). Since the coefficients of $\tilde{P}$ belong to $\B_K^{\dagger,r}$, they are locally analytic for the action of $\Gamma_K$ as elements of $\Bt^{[r,r]}$ by lemma \ref{lemm locana is the same}. Thus there exists $k \gg 0$ such that for $g \in G_k$, we have that the coefficients of $gP$ are analytic functions of $G_k$. Moreover, we have the equality $(g\tilde{P})(g(\tilde{y})=0$ and $\tilde{P}'(\tilde{y})$ is invertible in $\Bt_K^{[r,r]}$ so that $\tilde{y} \in (\Bt_K^{[r,r]})^{\Gamma_K-\la}$ by the implicit function theorem for analytic functions (which follows from the inverse function theorem given on page 73 of \cite{serre2009lie}). Using once again lemma \ref{lemm locana is the same}, this shows that $\tilde{y} \in (\At_K^\dagger)^{\Gamma_K-\la}$ and thus $w = \tilde{y}$ satisfies the claim.
\end{proof}

\begin{coro}
\label{coro s(v) is unif}
In definition \ref{def of lift of unif of A_K}, $s(v)$ is actually a uniformizer of $\E_K$.
\end{coro}
\begin{proof}
Let $w$ be as in proposition \ref{prop v actually uniformizer upto phi}. Since we assumed at the beginning of the section that $K/\Qp$ is Galois, we can find $\tau \in \Gal(K/\Qp)$ whose image in $\Gal(k_K/\F_p)$ is the absolute Frobenius $\phi$. We let $\iota_\tau : \At = \O_K \otimes_{\O_{K_0}} W(\Et) \ra \At$ be the map defined by $(\tau \otimes \phi)$. Note that this map preserves locally analytic vectors, but that there is a shift in terms of ``level of analyticity'' coming from lemma \ref{lemm phi shift locana}.

We have $\iota_\tau^{-k}(w) \in (\At_K^\dagger)^{\la}$ and thus by corollary \ref{coro Atdaggerla in phi-infAKdagger} $\iota_\tau^{-k}(w) \in \phi_q^{-\ell}(\A_K^\dagger)$ for some $\ell \geq 0$. Therefore there exists $r > 0$ and $R(T) \in \cal{A}^{\dagger,r}[1/T]$ such that $\iota_\tau^{-k}(w) = \phi_q^{-\ell}(R(v))$. 

We also know by proposition \ref{prop v actually uniformizer upto phi} that $\iota_\tau^{-k}(w)$ lifts a uniformizer $u$ of $\E_K$, so that if $\overline{R} \in k_K(\!(T)\!)$ denotes the Laurent series obtained by reducing the coefficients of $R$ modulo $p$, we have $u = \phi_q^{-\ell}(\overline{R}(s(v)))$. Since $s(v) \in \E_K^+$ and since $u$ is a uniformizer of $\E_K$, there exists $f(T) \in k_K[T]$ such that $s(v) = f(u)$. This means that we have the inclusions
$$ k_K(\!(\phi_q^{\ell}(u))\!) \subset k_K(\!(s(v))\!) \subset k_K(\!(u)\!)$$
and thus since the extension $k_K(\!(u)\!)/k_K(\!(\phi^{\ell}(u))\!)$ is purely inseparable, so is $k_K(\!(u)\!)/k_K(\!(s(v))\!)$. This means that there exists $h \geq 0$ such that $\phi^{-h}(s(v))$ is a uniformizer of $k_K(\!(u)\!)$. 

But now $s\left(\iota_\tau^{-h}(\At_K^{\dagger})^{\Gamma_K,\lambda-\an}\right) \subset k_K(\!(\phi^{-h}(s(v)))\!) = \E_K$ by lemma \ref{lemm s(ana) is power series in v}, and thus $\lambda' = \lambda-h$ satisfies proposition \ref{prop exists lambda included in FON}. However, this contradicts our choice of $\lambda$, so that $h = 0$ and $s(v)$ is a uniformizer of $\E_K$.
\end{proof}

\begin{rema}
\label{rema locana implies lift}
In particular, corollaries \ref{coro la implies lift} and \ref{coro s(v) is unif} show that the existence of a nontrivial locally analytic vector implies the existence of an overconvergent lift of the field of norms as defined in \S \ref{section OC lift}. Note that this only holds \textit{a priori} for $\Zp$-extensions, because super-Hölder vectors in this case recover exactly the perfectization of the corresponding field of norms. As pointed out in remark 2.2.4 of \cite{berger2022super}, as soon as $K_\infty/K$ is a $p$-adic Lie extension whose Galois group is of dimension (as a $p$-adic Lie group) at least $2$, then the set of super-Hölder vectors of $\Et_K$ contains the field of norms $X_K(L_\infty)$ of any $p$-adic Lie extension $L_\infty/K$ contained in $K_\infty$ and is thus no longer generated by a single element over $k_K$.
\end{rema}

The following theorem summarizes most of the results of the section (note that $\alpha = 1$ by corollary \ref{coro s(v) is unif}):

\begin{theo}
\label{theo locana implies oc lift}
Let $K_\infty/K$ be a totally ramified $\Zp$-extension, and assume that $(\At_K^\dagger)^{\Gamma_K-\la} \neq \O_K$. Then there exists $\lambda \in \R_{\leq 0}$ and $r > 0$ such that for $s \geq r$, $(\At_K^{\dagger,r})^{\Gamma_K,\lambda-\an} \simeq \cal{A}_K^{\dagger,r}$. Moreover, we have $(\At_K^{\dagger})^{\Gamma_K-\la} = \phi_q^{-\infty}(\A_K^{\dagger})$. 
\end{theo}

\section{The operator $\nabla$ and the kernel of the maps $\theta \circ \phi_q^{-n}$}

In what follows, we still assume that $K_\infty/K$ is a $\Zp$-extension, so that it is abelian and by local class field theory, there exists a Lubin-Tate extension $K_{\LT}/K$ such that $K_\infty \subset K_{\LT}$. We let $\Gamma_{\LT} = \Gal(K_\LT/K)$ and $H_{\LT} = \Gal(\Qpbar/K_\LT)$ and we keep the notations from \S 1 and from the previous section. 

In particular, there exists $n \geq 0$ and $v \in \At_K^{\dagger,r_n}$ such that $v$ lifts a uniformizer of the field of norms of $K_\infty/K$ and is a locally analytic vector of $\At_K^{\dagger,r_n}$ for $\Gamma_K$. Since $v \in \At^{\dagger,r_n}$ and since $\theta : \At^{\dagger,r_0} \ra \O_{\Cp}$ is well defined, we can consider $v_m:=\theta \circ \phi_q^{-m}(v)$ for all $m \geq n$. By lemma \ref{lemma theta circ phi-n is same as mod c} and proposition \ref{prop FON embedding}, we have $v_K(\phi_q^{-n}(v)) \ra 0$ when $n \ra +\infty$, so up to increasing $n$ we can always assume that for all $m \geq n$, $v_{\E}(\theta \circ \phi_q^{-m}(v)) < c$ where $c = c(K_\infty/K)$ is as in \S \ref{section OC lift}, and we do so in what follows. In particular, $v_K(v_m) = \frac{1}{q^m}$ for $m \geq n$ by lemma \ref{lemma theta circ phi-n is same as mod c}.

\begin{prop}
\label{prop proan belongs to Robba}
If $s \geq r_n$, then a power series $R(T) = \sum_{n \in \Z}a_nT^n$, $a_n \in K$ is such that $R(v) \in \Bt_{\rig,K}^{\dagger,s}$ if and only if $R(T) \in \cal{R}_K^s$. 
\end{prop}
\begin{proof}
This follows directly from the proof of proposition \ref{prop locana belongs to surconvring}, using the same arguments as in proposition 7.5 and 7.6 of \cite{colmez2008espaces}. 
\end{proof}

In what follows, if $s \geq r_n$, we write $\B_{\rig,K}^{\dagger,s}$ for the set of $f(v)$, where $f \in \cal{R}_K^s$, and we let $\B_{\rig,K}^\dagger = \bigcup_{s \geq r_n}\B_{\rig,K}^{\dagger,s}$.

Since $v \in (\At_K^{\dagger,r_n})^{\Gamma_K,\lambda-\an}$, it is a pro-analytic element of $\Bt_{\rig,K}^{\dagger,r_n}$ for the action of $\Gamma_K$ by proposition \ref{prop la in At = pa in Btrig}. The operator $\nabla:= \log g$, for $g \in \Gamma_K$ close enough to $1$ is well defined on $(\Bt_{\rig,K}^{\dagger,r_n})^{\Gamma_K-\pa}$ so that $\nabla(v) \in (\Bt_{\rig,K}^{\dagger,r_n})^{\Gamma_K-\pa}$. If $\gamma \in \Gamma_K$ is a topological generator, then we also have
\[ \nabla(v)= \lim\limits_{n \ra +\infty}\frac{\gamma^{p^n}(v)-v}{p^n}. \]
In particular, by proposition \ref{prop proan belongs to Robba} the sequence $\frac{F_{\gamma^{p^n}}(T)-T}{p^n}$ converges in $\cal{R}_K^s$ for $s \geq r_n$ to an element $H(T)$ such that $H(v) = \nabla(v)$. 

We can rewrite $\frac{F_{\gamma^{p^n}}(T)-T}{p^n} = (F_{\gamma}(T)-T)\prod_{k=1}^n\frac{1}{p}\left(\frac{F_{\gamma^{p^k}}(T)-T}{F_{\gamma^{p^{k-1}}}(T)-T}\right)$. Since $F_{\gamma}(T)-T$ belongs to $\cal{A}_K^{\dagger,r}$ and is nonzero, it is invertible in the Robba ring $\cal{R}_K := \cup_{s > 0}\cal{R}_K^s$, and the convergence of the sequence $\frac{F_{\gamma^{p^n}}(T)-T}{p^n}$ in $\cal{R}_K$ thus implies the convergence in $\cal{R}_K$ of the infinite product 

\[\prod_{k \geq 1}\frac{1}{p}\left(\frac{F_{\gamma^{p^k}}(T)-T}{F_{\gamma^{p^{k-1}}}(T)-T}\right).\]
Let us write $H_k(T):= \frac{1}{p}\left(\frac{F_{\gamma^{p^k}}(T)-T}{F_{\gamma^{p^{k-1}}}(T)-T}\right)$, so that $p\cdot H_k \in \cal{A}_K^{\dagger,r}$ by lemma \ref{lemma f circ f divisible}. The convergence in $\cal{R}_K$ of the infinite product above is equivalent to the fact that, for $s \gg 0$, we have $H_k(T) \ra 1$ when $k \ra +\infty$ for $|\cdot|_s$.

\begin{lemm}
\label{lemma hatKinfty Gammam-an}
For $m \geq 1$, we have $(\hat{K_\infty})^{\Gamma_m-\an} = K_m$.
\end{lemm}
\begin{proof}
One can follow the first part of the proof of \cite[Thm. 3.2]{Ber14SenLa}.
\end{proof}

\begin{lemm}
\label{lemma vm in OKm upto ell}
There exists $n_0 \geq n$ and $\ell \geq 0$ such that for all $m \geq n_0$, $K(v_m) = K_{mh+\ell}$. 
\end{lemm}
\begin{proof}
For $m \geq n$, we let $L_m = K(v_m)$ be the extension of $K$ generated by $v_m$. Since $v$ is locally analytic for the action of $\Gamma_K$ and since $\theta$ and $\phi_q$ are $\G_K$-equivariant, we get that the $v_m$ are algebraic over $K$ by lemma \ref{lemma hatKinfty Gammam-an} and that $L_m \subset K_\infty$. Let $L = \cup_{m \geq n}L_m \subset K_\infty$. We first prove that $L=K_\infty$, which is equivalent to the fact that an element of $\Gamma_K$ acting trivially on $L$ is trivial. Let $g \in \Gamma_K$ be such that $g_{|L_m} = \mathrm{id}_{L_m}$ for all $m \geq n$. Then by definition of $L_m$, we have $g(v_m) = v_m$ for all $m \geq n$. Thus the power series $F_g(T)-T \in \cal{A}_K^{r_n}$ admits infinitely many zeroes in the open unit disc, namely the $v_{k}$, $k \geq n$ (since $|v_{k}| \ra 1$) and is therefore zero (since it is bounded). We thus obtain that $F_g(T) = T$ hence $g(v) = v$. Therefore $g$ acts as the identity on the field of norms of $K_\infty/K$ thus $g = \id$ in $\Gamma_K$ and we are done.

Now the inclusion $K \subset L_n$ induces a continuous injective morphism $\Gal(K_\infty/L_n) \subset \Gamma_K$ whose image is compact open and thus $K_\infty/L_n$ is a sub-$\Zp$-extension (totally ramified) of $K_\infty/K$. In order to prove the proposition, it thus suffices to prove that for $m$ big enough, $L_{m+1}/L_m$ is of degree $q$. 

Recall that $\phi_q(v)$ is an overconvergent series in $v$, so that there exists $Q(T) \in \cal{A}_K^{\dagger,r_{n+1}}$ such that $\phi_q(v) = Q(v)$. By definition of the elements $v_m$, this means that we have $Q^{\phi_q^{-1}}(v_{m+1})=v_m$, where $Q^{\phi_q^{-1}}$ is the series $Q$ where we have applied $\phi_q$ to the coefficients. Since $|v_m|_p \ra 1$ in $\Cp$, by the theory of Newton polygons, we have $v_K(v_{m+1}) = \frac{1}{q}v_K(v_m)$ (we have $Q(T) \equiv T^q \mod \mathfrak{m}_K$ by definition). Since $L_{m+1}/L_m$ has at most degree $q$ by the theory of Newton polygons and since $K_\infty/L_n$ is totally wildly ramified (since $K_\infty/K$ is as such), we have $v_{L_m}(v_m) = \frac{v_{L_{m-1}}(v_{m-1})}{q}[L_m:L_{m-1}]$ so that the sequence $(v_{L_m}(v_m))_{m \geq n}$ is nonincreasing. Since it is bounded below and has integers values, it is constant for $m$ big enough and the relation $v_K(v_{m+1}) = \frac{1}{q}v_K(v_m)$ for $m \gg 0$ implies that for $m \gg 0$, the extension $L_{m+1}/L_m$ is of degree $q$. 

Therefore, there exists $\ell \geq 0$ such that for $m$ big enough, we have $L_m = K_{mh+\ell}$. This proves the result.
\end{proof}

Recall that we can assume without loss of generality that $K/\Qp$ is Galois. If $F(T) = \sum_{n \in \Z}a_nT^n$ with the coefficients $a_n$ in $K$, we let for $\tau \in \Gal(K/\Qp)$ $F^\tau(T) = \sum_{n \in \Z}\tau(a_n)T^n$.

For $\tau \in \Gal(K/\Qp)$, we let $n(\tau) \in \{0,\ldots,h-1\}$ be such that $\tau$ acts by $\phi^{n(\tau)}$ on $k_K$. The map $h_\tau:=(\tau \circ \phi^{n(\tau)})$ is well defined on $\At = \O_K \otimes_{W(k_K)}\Et$ and on $\At^\dagger$ (note that $h_\tau$ induces an isomorphism between $\At^{\dagger,r}$ and $\At^{\dagger,p^{n(\tau)}r}$), and the latter extends into a map $\Bt^I \ra \Bt^{p^{n(\tau)}I}$.

For $\tau \in \Gal(K/\Qp)$, we let $y_\tau = h_\tau(u)$, where $u$ is the Lubin-Tate period for $\Gamma_{\LT}$ that has been defined in \S 3. We also let $t_\tau = h_\tau(t_\pi) = \log_{\LT}^\tau(y_\tau)$. If we let $Q_n^\tau = Q_n^\tau(y_\tau)$, then $t_\tau = y_\tau \prod_{n \geq 1}Q_n^\tau/\tau(\pi)$. Note that since $Q_n/\pi$ is a generator of $\ker(\theta \circ \phi_q^{-n} : \Bt_{\rig}^+ \ra \Cp)$, $Q_n^\tau/\tau(\pi) = h_\tau(Q_n/\pi)$ is a generator of $\ker(\theta \circ \phi_q^{-n} \circ h_\tau^{-1} : \Bt_{\rig}^+ \ra \Cp)$. When $K=\Qp$ and $\pi=p$, $t_p$ is exactly the classical $t$ of $p$-adic Hodge theory. We still denote its image by $t$ in $\Bt_{\rig}^+$.

\begin{lemm}
\label{lemma t prod t_tau}
There exists $a \in (K\cdot \hat{\Q_p^{\mathrm{nr}}})^\times$ such that $\prod_{\tau \in \Gal(K/\Qp)} = a\cdot t$ in $\Bt_{\rig}^+$.
\end{lemm}
\begin{proof}
This is basically proposition 3.4 of \cite{BMTensTriang} and then one has to use the fact that $\Bt_{\rig}^+$ is a subring of $\B_{\mathrm{cris}}^+$ and that every element involved belongs to $\Bt_{\rig}^+$. 
\end{proof}

\begin{lemm}
For any $\tau \in \Gal(K/\Qp)$, the element $\nabla(v)$ is divisible by $t_\tau$ in $\Bt_{\rig}^{\dagger,r_n}$.
\end{lemm}
\begin{proof}
Let $m \geq n$ and let $\tau \in \Gal(K/\Qp)$. Since the maps $\theta$ and $\phi_q$ are $\G_K$-equivariant, we have that 
\[ \theta \circ \phi_q^{-n} \circ h_\tau^{-1}(\nabla(v)) = \nabla(\theta \circ \phi_q^{-n} \circ h_\tau^{-1}(v)). \]
Note that $\theta \circ \phi_q^{-n} \circ h_\tau^{-1}(v)$ is a locally analytic vector for $\Gamma_K$ in $\hat{K_\infty}$, so that it belongs to $K_\infty$ by lemma \ref{lemma hatKinfty Gammam-an}, and thus is killed by $\nabla$. In particular, this implies that for any $m \geq n$ and for any $\tau \in \Gal(K/\Qp)$, $\nabla(v)$ belongs to $\ker(\theta \circ \phi_q^{-n} \circ h_\tau^{-1}) = (Q_n^\tau)$.

This means that for all $\tau \in \Gal(K/\Qp)$, $t_\tau | \nabla(v)$ in $\Bt_{\rig}^{\dagger,r_n}$ (the argument for the division by an infinite product works as in lemma 4.6 of \cite{Ber02}).
\end{proof}

\begin{coro}
\label{coro t divides nabla}
The element $\nabla(v)$ is divisible by $t$ in $\Bt_{\rig}^{\dagger,r_n}$.
\end{coro}
\begin{proof}
We just combine the two previous lemma. 
\end{proof}

We let $z = \frac{\nabla(v)}{t} \in \Bt_{\rig}^{\dagger,r_n}$ and $\lambda := \frac{\phi_q(z)}{z} = \frac{Q'(v)}{q}$.

\begin{lemm}
\label{lemma lambda in AtKdagger}
We have $\lambda \in (\A_K^{\dagger})^\times$.
\end{lemm}
\begin{proof}
If $z \in \Bt^\dagger$, up to replacing $z$ by $\pi^nz$ for some nonnegative $n$ (which does'nt change the value of $\frac{\phi_q(z)}{z}$), we can assume that $z \in \At^\times$, and then $\lambda = \frac{\phi_q(z)}{z} \in \At \cap \B_K^{\dagger} = \A_K^{\dagger}$.

If $z \in \Bt_{\rig}^{\dagger}$ but not in $\Bt^\dagger$, then $r \mapsto V(z,r)$ is eventually decreasing as $r \rightarrow +\infty$ by corollary \ref{coro function bounded or eventually increasing}. Moreover, we have $V(\phi_q(z),qr) = V(z,r)$ so that $V(\lambda,qr) = V(z,r)-V(z,qr) \geq 0$ for $r \gg 0$. By lemma \ref{lemm when is function bounded}, this implies that $\lambda \in \Bt^{\dagger}$.

Now let us write $\lambda = \sum_{n \in \Z}\lambda_nv^n$, $\lambda_n \in K$. Since $V(\lambda,r) \geq 0$ for $r \gg 0$, we have for all $n$ that $|\lambda_n|\rho^n \leq 1$ for $\rho$ close enough to $1$, and thus $\lambda_n \in \O_K$ for all $n$ so that $\lambda \in \A_K^\dagger$.

In both cases, we have $\lambda \in \A_K^\dagger$, and since $Q(v) = v^q \mod \pi$, we have $s(\frac{Q'(v)}{q}) > 0$ (because the coefficients in $s(v)^{q-1}$ is $1$) and so by lemma \ref{lemma v in r not just rho} we have $\lambda \in (\A_K^{\dagger})^\times$.
\end{proof}

\begin{prop}
\label{prop nabla/t inv}
The element $\frac{\nabla(v)}{t}$ is invertible in $\Bt_{\rig}^\dagger$.
\end{prop}
\begin{proof}
By lemma \ref{lemma lambda in AtKdagger} we can apply lemma \ref{lemma phi(x)/x means étale} to $z$, which shows that $\frac{\nabla(v)}{t}$ is an invertible element of $\Bt^\dagger$.
\end{proof}

For $\tau \in \Gal(K/\Qp)$ and $m \geq n$, we let $v_m^\tau:= \theta \circ \phi_q^{-m} \circ h_\tau^{-1}(v)$. Since $v$ is locally analytic for the action of $\Gamma_K$, so are the $v_m^\tau$ and so they belong to $K_\infty$ by lemma \ref{lemma hatKinfty Gammam-an}.

\begin{prop}
\label{prop exists gen kertheta}
For each $\tau \in \Gal(K/\Qp)$ and $m \geq n$, the minimal polynomial $\mu_{K,v_m^\tau}$ of $v_m^{\tau}$ over $K$ is such that $\mu_{K,v_m^\tau}(v)$ is a generator of $\ker(\theta \circ \phi_q^{-m} \circ h_\tau^{-1} : \Bt_{\rig}^{\dagger,r_n} \ra \C_p)$. 
\end{prop}
\begin{proof}
The map $\theta \circ \phi_q^{-m} \circ h_\tau^{-1}: \B_K^{\dagger,r_n} \ra K(v_m^\tau)$ is surjective by definition, and its kernel is a principal ideal generated by an element of $K[T]$ by proposition \ref{prop Lazard PID}. Since $K(v_m^\tau)$ is a field, this ideal is maximal, and so its monic generator is an irreducible element of $K[T]$, which by definition has $v_m^\tau$ as a root, and thus is equal to $\mu_{K,v_m^\tau}$.

Since for all $\tau \in \Gal(K/\Qp)$ and $m \geq n$, $\theta \circ \phi_q^{-m} \circ h_\tau^{-1}(H(v)) = H(v_m^\tau)=0$, we get that $\mu_{K,v_m^\tau}(v) | H(v)$ in $\Bt_{\rig}^{\dagger}$. Moreover, by definition, the ideal generated by $\mu_{K,v_m^\tau}(v)$ in $\Bt_{\rig}^{\dagger}$ is contained in $(Q_n^\tau)$. Since $H(v) | \prod_{\tau \in \Gal(K/\Qp),m \geq n} \frac{Q_n^\tau}{\tau(\pi)}$, this means that the ideals $(Q_n^\tau)$ and $(\mu_{K,v_m^\tau}(v))$ coincide in $\Bt_{\rig}^{\dagger}$. 
\end{proof}

\begin{prop}
\label{prop implies exists lambda1 and lambda2}
For each $\tau \in \Gal(K/\Qp)$ there exists $S_{\tau}(T) \in \cal{R}_K$ such that $\frac{S_{\tau}(v)}{t_\tau} \in \Bt_{\LT}^{\dagger}$, and $\frac{1}{\nabla(v)}\prod_{\tau \in \Gal(K/\Qp)}S_\tau(v) \in (\B_{\rig,K}^\dagger)^\times$.
\end{prop}
\begin{proof}
Proposition \ref{prop exists gen kertheta}, along with the fact that $H(v)$ is equal, up to a unit of $\Bt_{\rig}^{\dagger}$, to $\prod_{\tau \in \Gal(K/\Qp),m \geq n} \frac{Q_n^\tau}{\tau(\pi)}$, itself equal up to a unit to $\prod_{\tau \in \Gal(K/\Qp)}t_\tau$ implies that $H(T)$ admits a decomposition $H(T) = \prod_{\tau \in \Gal(K/\Qp)}S_\tau(T)$, where $S_{\tau}(T) \in \cal{R}_K$ and $S_\tau(v)/t_\tau$ is a unit of $\Bt_{\rig}^{\dagger}$.
\end{proof}

\section{The anticyclotomic setting and a conjecture of Berger}

Berger's conjecture (\cite[Conj. A]{berger2022substitution}) states that, given a Robba ring $\cal{R}_K$ over a finite extension $K$ of $\Qp$, endowed with an overconvergent Frobenius map $\phi_q$ \footnote{Berger's conjecture is stated more generally for overconvergent substitution maps, not just Frobenius maps.}, we have $(\mathrm{Frac}\cal{R}_K)^{\phi_q=1} = K$, or equivalently, that for each $\lambda \in \cal{R}_K$, the ``eigenspaces'' $\cal{R}_K^{\phi_q=\lambda}$ have at most dimension $1$ as $K$-vector spaces. We first explain why this is relevant in the anticyclotomic setting, and then shall prove that the eigenspace $\cal{R}_K^{\phi_q=Q'}$ has $K$-dimension $1$ in our setting.

Let us assume that $K=\Q_{p^2}$ and that $K_\infty/K$ is the anticyclotomic extension. We let $K_{\LT}$ be the Lubin-Tate extension of $K$ attached to the uniformizer $p$. Recall that $\sigma \in \Gal(K/\Qp)$ is the absolute Frobenius. We suppose that there exist nontrivial locally analytic vectors in $\At_K^\dagger$ in the anticylotomic case, and keep the notations from the previous sections. Note that, in the anticyclotomic setting, $\frac{t_\id}{t_\sigma} \in (\mathrm{Frac}\Bt_{\rig,K}^\dagger)^{\phi_q=1}$.

By proposition \ref{prop implies exists lambda1 and lambda2}, there exist $S_\id \in \cal{R}_K$ such that $\alpha:=\frac{S_\id(v)}{t_\id} \in \Bt_{\LT}^{\dagger}$. We let $\lambda_1:=S_\id$ . Recall that we wrote $H(T) \in \cal{R}_K$ for the power series such that $\nabla(v) = H(v)$.

\begin{lemm}
We have $\frac{\phi(\alpha)}{\alpha} \in \bigcup_{m \geq 0}\phi_q^{-m}(\B_{K}^{\dagger})$.
\end{lemm}
\begin{proof}
We have $\alpha = \frac{S_\id(v)}{t_\id} \in \Bt_{\LT}^{\dagger}$. The same proof as in lemma 5.1.1 of \cite{GP18} shows that $\alpha$ is a pro-analytic vector of $\Bt_{\rig,\LT}^{\dagger}$ for the action of $\Gal(K_\LT/K)$. Since $\phi$ commutes with the Galois action, so is $\phi(\alpha)=\frac{\phi(S_\id(v))}{t_\sigma}$.  

Writing 
\[\frac{\phi(\alpha)}{\alpha} = \frac{\phi(\lambda_1(v))}{\lambda_1(v)}\cdot \frac{t_\id}{t_\sigma} \]
shows that $\frac{\phi(\alpha)}{\alpha}$ is invariant by $\Gal(K_\LT/K_\infty)$ (since it is the case of $\lambda_1(v), \phi(\lambda_1(v))$ and $\frac{t_\id}{t_\sigma}$).  Therefore, we can conclude by proposition \ref{prop la in At = pa in Btrig} and corollary \ref{coro Atdaggerla in phi-infAKdagger}.
\end{proof}

We can assume, up to replacing $\lambda_1$ by $\phi_q^m(\lambda_1)$ for $m \gg 0$, that $\frac{\phi(\alpha)}{\alpha} \in \B_{K}^{\dagger}$, and we let $\lambda_2 \in \cal{R}_K$ be the element such that $\lambda_2(v):= \alpha t_\sigma$ (which exists since $\alpha t_\sigma = \frac{\alpha}{\phi(\alpha)}\cdot \phi(\lambda_1(v))$). Since $\frac{t_\id}{t_\sigma} \in (\mathrm{Frac}\Bt_{\rig,K}^\dagger)^{\phi_q=1}$, this means that $\frac{\lambda_1}{\lambda_2} \in (\mathrm{Frac}\cal{R}_K)^{\phi_q=1}$, and so is a potential counterexample to Berger's conjecture. Moreover, we have by proposition \ref{prop nabla/t inv} that $\frac{\lambda_1\lambda_2}{H}$ is invertible in $\cal{R}_K$. 

We let $\beta = \frac{\phi_q(\lambda_1)}{\lambda_1} = \frac{\phi_q(\lambda_2)}{\lambda_2} \in \cal{R}_K$, and we let $W := W(\lambda_1,\lambda_2)=\lambda_1'\lambda_2-\lambda_1\lambda_2'$ be the Wronskian of $\lambda_1$ and $\lambda_2$. 

\begin{lemm}
We have $W \in K^\times\cdot(\cal{A}_K^{\dagger})^\times$. Moreover, $\frac{\lambda_1\lambda_2}{W}, \frac{\lambda_1^2}{W}, \frac{\lambda_2^2}{W} \in \cal{R}_K^{\phi_q = Q'}$. In particular, there exists $\ell \in K \neq 0$ such that $\frac{\lambda_1\lambda_2}{W} = \ell\cdot H$.
\end{lemm}
\begin{proof}
Derivating the equality $\lambda_i \circ Q = \beta \lambda_i$ for $i=1,2$ yields $W \in \cal{R}_K^{\phi_q = \beta^2/Q'}$. Note that $\beta = \frac{\phi_q(\lambda_1)}{\lambda_1}$ and that 
\[\frac{\phi_q(\lambda_1(v))}{\lambda_1(v)} = \frac{\phi_q(\alpha)}{\alpha}\frac{\phi_q(t_\id)}{t_\id} = p \frac{\phi_q(\alpha)}{\alpha}.\]
In particular, by lemma \ref{lemma phi(x)/x means étale}, this means that $v_p(\beta) = 1$, where $v_p(f) = \inf(v_p(a_i))$ for $f = \sum_{i \in \Z}a_iT^i$. Therefore, $v_p(\beta^2/Q') = 0$ by lemma \ref{lemma lambda in AtKdagger} and so belongs to $(\cal{A}_K^{\dagger})^\times$, which implies by lemma \ref{lemma phi(x)/x means étale} that $W$ is, up to multiplication by a nonzero element of $K$, an element of $(\cal{A}_K^{\dagger})^\times$, so that the three quotients in the statement of the lemma make sense, and belong to $\cal{R}_K^{\phi_q = Q'}$.  

Since $H$ is divisible by $\lambda_1\lambda_2$ and satisfies $\phi_q(H) = Q' \cdot H$, we get that 
\[\frac{\lambda_1\lambda_2}{WH} \in \cal{R}_K^{\phi_q=1} = K.\]
\end{proof}

We let $f_i = \frac{\lambda_i^2}{W}$ for $i=1,2$. 

\begin{lemm}
\label{lemma phi invar means galois invar}
There exists $m \geq 0$ such that $f_1 \circ F_g = F_g' \cdot f_1$ for all $g \in \Gamma_m$.
\end{lemm}

\begin{proof}

Let $g \in \Gamma_K$ and let $i$ be either $1$ or $2$. Let $h(T):= \frac{f_i \circ F_g}{F_g'}$. We first prove that $h \circ Q = Q' \cdot h$.

We have $h \circ Q = \frac{(f_i \circ Q) \circ F_g}{F_g' \circ Q}$ since $F_g$ and $Q$ commute, and so 

\[ h \circ Q = \frac{(Q' \cdot f_i) \circ F_g}{F_g' \circ Q} = \frac{Q' \circ F_g}{F_g' \circ Q}\cdot f_i \circ F_g \]

and derivating the equality $F_g \circ Q = Q \circ F_g$ yields $\frac{Q' \circ F_g}{F_g' \circ Q} = \frac{Q'}{F_g'}$ so that 

\[ h \circ Q = Q' \cdot h. \]

Since $\frac{f_i \circ F_g}{F_g'}$ is divisible by $f_i$ in $\cal{R}_K$, this means that $\frac{f_i \circ F_g}{f_i \cdot F_g'} \in \cal{R}_K^{\phi_q=1}$ and so there exists some $\ell_i(g)$ in $K$ (which has to be nonzero) such that $f_i \circ F_g = \ell_i(g) F_g' \cdot f_i$. 

Derivating the equality $f_i \circ F_g = \ell_i(g) F_g' \cdot f_i$ yields
\[ F_g'(T)\cdot f_i' \circ F_g(T) = \ell_i(g)\cdot(F_g'(T)f_i'(T)+F_g''(T)f_i(T)). \]

Note that, for example from the fact that $\frac{t_\id}{\lambda_1(v)}$ is an invertible element of $\Bt_{\LT}^\dagger$, we know that $f_1(\theta \circ \phi_q^{-m}(v))= 0$ for $m \geq n$, and so if $g_m$ is a topological generator of $\Gal(K_\infty/K(v_m))$, we obtain $\ell_1(g_m) = 1$. Therefore, we obtain that for all $g \in \Gal(K_\infty/K(v_m))$, $f_1 \circ F_g = F_g' \cdot f_1$, which proves the claim.
\end{proof}

\begin{coro}
We have $\frac{\lambda_1}{\lambda_2} \in K$.
\end{coro}
\begin{proof}
By lemma \ref{lemma phi invar means galois invar}, we have that 
\[\frac{\lambda_1(v)}{\lambda_2(v)} = \frac{\lambda_1^2(v)}{W(v)}\cdot \frac{W(v)}{\lambda_1(v)\lambda_2(v)} = \frac{\lambda_1^2(v)}{W(v)}\frac{1}{H(v)} \]
is such that $g(\frac{\lambda_1(v)}{\lambda_2(v)}) = \frac{\lambda_1(v)}{\lambda_2(v)}$ for all $g \in \Gamma_m$ for some $m \gg 0$. Therefore, $\frac{\lambda_1(v)}{\lambda_2(v)} \in (\mathrm{Frac}\Bt_{\rig,K}^\dagger)^{\G_{K_m}} = K$ by corollary 28.7 of \cite{Ber10IHP}, which proves the claim.
\end{proof}

In particular, this proves that the potential counterexample to Berger's conjecture does not exist, and so proves the following:

\begin{theo}
\label{theo anticyclo nolift}
There is no overconvergent lift of the field of norms of $K_\infty/K$ when $K=\Q_{p^2}$ and $K_\infty/K$ is the anticyclotomic extension.
\end{theo}

\section{Kedlaya's conjecture and higher locally analytic vectors}
We now explain how to use the results from the previous sections to refute Kedlaya's conjecture.

\begin{prop}
\label{prop if no locana then no kedlaya}
Let $K_\infty/K$ be a $\Zp$ extension with Galois group $\Gamma_K$, and assume that $(\At_K^{\dagger})^{\Gamma_K-\la} = \O_K$. Then Kedlaya's conjecture is false for $K_\infty/K$.
\end{prop}
\begin{proof}
Let us assume that $(\At_K^{\dagger})^{\Gamma_K-\la} = \O_K$ and that Kedlaya's conjecture is true. This means that if $T$ is a free $\O_K$-representation of $\G_K$ then $\D_K^{\dagger,\an}(T):=(\At^\dagger \otimes_{\Zp}T)^{H_K,\Gamma_K-\la}$ is an $\O_K$-module such that $\At^\dagger \otimes_{\O_K} \D_K^{\dagger,\an}(T) \simeq \At^\dagger \otimes_{\Zp}T$ and thus $(\At^\dagger \otimes_{\O_K} \D_K^{\dagger,\an}(T))^{\phi_q=1} \simeq T$. 

Moreover, since $(\At_K^{\dagger})^{\Gamma_K-\la} = \O_K$, we can assume that $K_\infty/K$ is not a twisted cyclotomic extension of $K$. Now let $T$ be a rank $1$ $\O_K$-representation of $\G_K$, with basis $e$. By Kedlaya's conjecture, there exists $y \in \At^{\dagger}$ such that $(T \otimes_{\O_K}\At^{\dagger})^{H_K,\Gamma_K-\la}$ is a rank $1$ $\O_K$-module generated by $e \otimes y$, and comes equipped with an $\O_K$-linear action of $\Gamma_K$ and $\phi_q$. In particular, there exists $a \in \O_K^\times$ (since $\phi_q$ is an isomorphism) such that $\phi_q(e \otimes y) = a \cdot (e \otimes y)$, and $\Gamma_K$ acts on $e \otimes y$ by multiplication by some character $\eta : \Gamma_K \ra \O_K^\times$. 

By local class field theory, there exists $z$ in $\O_{\hat{K^{\mathrm{unr}}}}$, the ring of integers of the $p$-adic completion of the maximal unramified extension of $K$, such that $\frac{z}{\phi_q(z)}=a$. Since $\O_{\hat{K^{\mathrm{unr}}}} \subset \Atplus \subset \At^\dagger$, we have that $z \in \At^{\dagger}$ and if $x=e \otimes y \otimes z \in  \D_K^{\dagger,\an}(T)\otimes_{\O_K}\At^\dagger   \simeq T \otimes_{\O_K}\At^\dagger$, we get that  $\phi_q(x) = x$ so that $yz \in \At^\dagger$ is invariant by $\phi_q$ and thus belongs to $\O_K$. 

This means that $y \in \O_{\hat{K^{\mathrm{unr}}}}$, and since $\Gamma_K$ acts on $e \otimes y$ by multiplication by some character $\eta : \Gamma_K \ra \O_K^\times$, this means that $\G_K$ acts on $e$ by multiplication by a character which factors through $\Gal(K_\infty\cdot K^{\mathrm{unr}}/K)$. Since this is true for any rank $1$ representation $T$ of $\G_K$, this means by local class field theory that $K^{\mathrm{ab}} = K_\infty\cdot K^{\mathrm{unr}}$, which is possible if and only if $K_\infty$ is a Lubin-Tate extension of $K$. Since $K_\infty/K$ is a $\Zp$-extension, this means that $K=\Qp$ and that $K_\infty/K$ is an unramified twist of the cyclotomic extension of $K$, which as stated above is ruled out by the assumption that $(\At_K^{\dagger})^{\Gamma_K-\la} = \O_K$. 
\end{proof}

As a corollary of theorem \ref{theo anticyclo nolift} and proposition \ref{prop if no locana then no kedlaya}, we obtain the following theorem:

\begin{theo}
\label{theo anticyclo counters kedlayasconj}
The anticyclotomic extension provides a counterexample to Kedlaya's conjecture. 
\end{theo}

When $(\At_K^{\dagger})^{\Gamma_K-\la}=\O_K$, it is straightforward to see that this implies the existence of nontrivial locally analytic vectors attached to $\At_K^{\dagger}$-modules of rank $1$, by taking some well chosen exact sequence. One could think of the exact sequence 

\[0 \rightarrow \ker(\theta \circ \phi_q^{-k}) \rightarrow \At_K^{\dagger,r_k} \rightarrow \O_{\hat{K_\infty}} \rightarrow 0 \] 

but the problem is that $\O_{\hat{K_\infty}}$ is not a Tate ring so that this exact sequence is not in the right category (thanks to Gal Porat for pointing this out). If one tries to solve the problem by inverting $p$, then $\hat{K_\infty}$ is a $p$-adic Banach space but $\B_K^{\dagger,r_k}$ is not, so that the exact sequence obtained after inverting $p$ in the sequence above is not an exact sequence in the category of $p$-adic Banach spaces. 

However, one can take locally analytic vectors in the sequence 

\[ 0 \rightarrow \At_K^{\dagger} \xrightarrow{\pi} \At_K^{\dagger} \rightarrow \Et_K \rightarrow 0 \]

where every object is a Tate ring, and so this sequence gives rise to nontrivial higher locally analytic vectors.

\bibliographystyle{amsalpha}
\bibliography{bibli}
\end{document}